%% file: main_interpolation.tex
\documentclass{article}

\usepackage{amssymb,amsthm,amsmath}
\usepackage{enumerate,bm,mathrsfs}
\usepackage{nicematrix}
\usepackage{tabularx}
\usepackage{svg}
\usepackage{tikz,graphicx}
\usepackage{tikz-network}
\usepackage{tikz-3dplot,pgfplots}
\usetikzlibrary{math}

\usetikzlibrary{backgrounds}
\usetikzlibrary{arrows}
\usetikzlibrary{shapes,shapes.geometric,shapes.misc}
\tikzstyle{tikzfig}=[baseline=-0.25em,scale=0.5]
\pgfdeclarelayer{edgelayer}
\pgfdeclarelayer{nodelayer}
\pgfsetlayers{background,edgelayer,nodelayer,main}
\usepackage{tikzscale}

\usepackage{caption, subcaption}
\usepackage{geometry}
\usepackage{hyperref}
\usepackage{cleveref}

\usepackage{geometry}
\providecommand{\keywords}[1]
{
  \small
  \textbf{\textit{Keywords---}} #1
}

\providecommand{\msc}[1]
{
  \small
  \textbf{\textit{MSC---}} #1
}

\pgfplotsset{compat=1.10}
\usetikzlibrary{angles}

\newtheorem{theorem}{Theorem}[section]
\newtheorem{corollary}[theorem]{Corollary}
\newtheorem{lemma}[theorem]{Lemma}
\newtheorem{proposition}[theorem]{Proposition}

\theoremstyle{definition}
\newtheorem{example}[theorem]{Example}
\newtheorem{definition}[theorem]{Definition}
\newtheorem{remark}[theorem]{Remark}

\input{macros}

\title{Interpolation in path-connected metric spaces using splines on embedded graphs}

\author{
E. Fuselier\thanks{High Point University, High Point, NC.}
\and
J. P. Ward\thanks{North Carolina A\&T State University, Greensboro, NC.}
\thanks{Corresponding author, \url{jpward@ncat.edu}}
\and
J. Zhao\thanks{University of Minnesota, Minneapolis, MN.}
}

\date{\today}

\begin{document}

\maketitle

\begin{abstract}
We study
a local graph-based interpolation framework for data defined on graph vertices embedded in a metric space. The resulting interpolants are constructed from variational graph splines defined on small neighborhoods, yielding continuous-domain extensions of the graph data. We analyze their behavior through energy estimates, Lipschitz continuity, and stability properties, and compare them with $k$-Nearest Neighbor ($k$NN) interpolation
as a natural baseline due to its simplicity and widespread use.
Numerical experiments on the manifold of symmetric positive definite matrices illustrate the relative performance of the proposed local interpolation schemes and $k$NN interpolation across varying levels of smoothness and noise.
\end{abstract}


\msc{05C90, 41A05, 41A15, 65D07}

\section{Introduction}

Graph-based models have become an important tool in a wide range of applications, including neural networks, image processing, and chemical analysis, to name a few \cite{hariri2024graph,leus2023graph,nallapareddy2024towards,shuman2013emerging}. When the graph vertices are embedded in a metric space, one may want to use the graph information to explore a quantity of interest in the surrounding space. In this paper, we study a simple local interpolation scheme for this setting. Such interpolants may be useful in semi-supervised learning \cite{belkin2006manifold,zhu2003semi}, inpainting and other signal reconstruction problems in graph signal processing \cite{ortega2017graph}, and optimization of expensive black-box functions on manifolds, where local surrogate models are often employed \cite{vu2017surrogate}.

The procedure studied here is based on variational graph splines \cite{pesenson2008sampling,pesenson2009,ward2020}, which are closely related to kernel interpolation techniques on graphs \cite{erb2022_GBF}. For each evaluation point, we construct a small graph consisting of the point together with nearby graph vertices, and define the interpolated value as the graph spline that minimizes a discrete-derivative energy subject to the function values. As the evaluation point varies, these local constructions produce a continuous-domain interpolant in the embedding space. Note the similarity to the standard $k$-Nearest Neighbor ($k$NN) interpolation scheme \cite[Chapter 8]{learning1997tom}, which is often used to interpolate missing data in abstract settings. However, due to the incorporation of an energy minimization property, the interpolant described here tends to capture underlying trends in the data that may be missed by $k$NN. In this paper, we investigate several properties of the local graph interpolants, including their energy behavior, Lipschitz continuity, stability, and comparison with $k$NN interpolation.

The paper is organized as follows. After reviewing the necessary background, in \Cref{sec:interp} we describe the construction of local graph interpolants. The bulk of the paper is in \Cref{sec:properties}, where we first analyze the interpolants along geodesics, including an energy estimate for one of the interpolants. We then briefly justify the choice of graph weights used throughout the construction. Next, we investigate the Lipschitz behavior of the interpolants and their
$k$NN counterpart on path subgraphs through explicit formulas. We then establish stability estimates for the corresponding interpolation operators on general graphs. Finally, in \Cref{sec:experiments} we present numerical experiments
on the manifold of symmetric positive definite matrices that illustrate and compare the methods under varying levels of smoothness and noise. The experiments suggest that each has potential advantages, depending on the smoothness of the underlying data and the level of noise that is present.

\section{Preliminaries}
\interfootnotelinepenalty=10000

In this section, we recall some standard definitions and establish the scene for the interpolation problem.

\subsection{Graph Prerequisites}

Consider an undirected, weighted graph $\graph$ with vertex set $\vertices$ and edge set $\edges$.
We let $A$ denote the weighted adjacency matrix and $D$ be the diagonal matrix where each diagonal entry is the corresponding row sum of $A$. We denote an $\abs{\edges}\times \abs{\vertices}$ signed incidence matrix by $M$, where row $e=(u,v)$ has a $1$ in column $u$ or $v$ and a $-1$ in the other column. The $\abs{\edges}\times \abs{\edges}$ edge weight matrix is denoted $W$; this is the diagonal matrix containing the weights of the edges. We denote the adjacency of vertices $u,v$ by $u\sim v$, and the degree of vertex $u$ is denoted
\begin{equation}\label{eq:vertex_degree_definition}
d_u
=
\sum_{v\sim u} \weight{u}{v},
\end{equation}
where $\weight{u}{v}$ is the weight of edge $(u,v)$.

Derivative matrices on graphs will be important in what follows. The graph Laplacian plays the role of a second-order derivative\footnote{This is somewhat misleading as the weights do not necessarily match those in the finite difference that is used to define the derivative in the underlying space.} on the graph. There are several conventions for the graph Laplacian. Possibly the most common is the \emph{standard Laplacian}, given by $\lap := D-A$. If $f:\vertices\rightarrow \mathbb{R}$ (i.e. a $|\vertices|\times 1$ vector), we
have the useful vertex-based representation
\begin{equation}\label{eq:lapf_expanded}
\lap f(u) = \sum_{v\sim u} (f(u)-f(v))\weight{u}{v}.
\end{equation}
Normalizations of the standard Laplacian include the \emph{normalized Laplacian} and the \emph{random walk Laplacian}, respectively given by
\begin{equation}\nonumber
\nlap: = D^{-1/2}\lap D^{-1/2}
\quad\mbox{and}\quad
\rwlap:=D^{-1}\lap.
\end{equation}
In this paper, we will focus on the Standard and random walk varieties.

Note that the product $\fodgrad:=WM$ acts as a first-order derivative on the graph: $\fodgrad f$ assigns a value to each edge that is an approximation of a first-order directional derivative. Another first-order derivative is given by $\fodlapfactor:=W^{1/2}M$. It is not hard to show (see \cite{benedetto2015graph}, for example) that this derivative satisfies the equation
\begin{equation}\nonumber
\fodlapfactor^T\fodlapfactor = (W^{1/2}M)^\top(W^{1/2}M)=\lap.
\end{equation}

We will view most of these derivatives as operators on $\ell_2(\vertices)$, where $\vertices$ can be inferred as the domain of the function that the derivative is acting on. However, for the random walk Laplacian it is convenient to work on the weighted space $\wspd$, with inner product defined by
\[\langle f, g\rangle_{\wspd} := f^T Dg.\]
Also, note that while the matrix
$\rwlap$ is not symmetric in general, $\rwlap:\wspd\rightarrow\wspd$ is a self-adjoint operator \cite[Lemma 2.8]{grigor18}.

\subsection{Metric Space Setting}
Let $(\metspace,\dist_\metspace)$ be a path-connected metric space. Our aim is to approximate a function $f$ living on this metric space, and we assume that function values are expensive to compute. Suppose that we know $f$ on some node set $\nodes=\set{x_i}_{i=1}^n\subset\metspace$, and that the nodes are vertices in a weighted graph $\graph$. Let $x\in \metspace$ be a point where an approximation to $f$ is desired.
To determine the interpolant at this location, we begin by incorporating $x$ with the graph $\graph$ to create a new graph $\graphx$ with vertex set $\localnodes\cup\{x\}$ (see \Cref{fig:graphx})\footnote{We are being somewhat vague here on purpose to accommodate the multitude of ways this could be accomplished. One might add edges from $x$ to its nearest neighbors from $X$, or, as we will consider later in the paper, we could delete an edge from $\graph$ and add edges from $x$ to each vertex of the deleted edge.}.
The interpolated value will then be determined by a graph spline interpolation procedure
(see \Cref{sec:interp}), entirely based on information from a local subgraph $\subgraphx$.

For this paper, we use the metric on $\metspace$ to determine the edge weights on the graphs. Specifically, the weight on edge $(u,v)$, is defined by the inverse of the distance between the vertices, i.e.
$\weight{u}{v}:=1/\dist_\metspace(u,v)$, a choice whose motivation is discussed in \Cref{subsec:weightanddistances}. This induces a metric on the graph as follows. The graph distance between adjacent vertices $u$ and $v$  is simply defined to be $\dist_\metspace(u,v)$.  If $u$ is not adjacent to $v$, the graph distance is the sum of the edge lengths on a shortest path in the graph. Like the metric $\dist_\metspace$, we use $\dist$ to denote the graph distance functions and use subscripts to indicate the underlying graph.
In most cases, we are inserting new points $x$ on geodesics of $\metspace$, which gives
$\dist_{\graph}=\dist_{\graphx}$, and this is necessary for some results. However, when we consider inserting $x$ off of a geodesic, these distances
may not coincide.

\input{figures/graph.tikzstyles}
\begin{figure}[ht!]
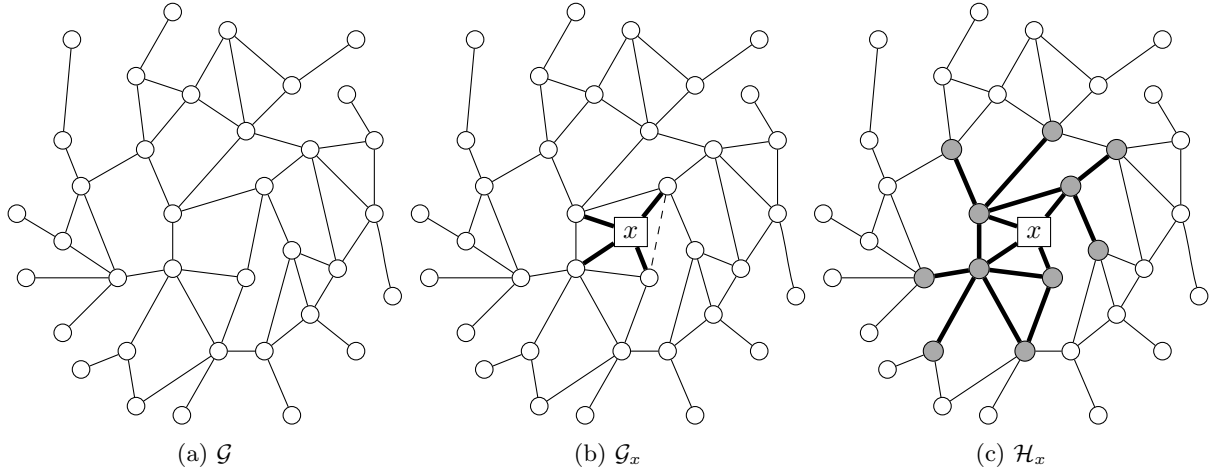

\centering
\begin{subfigure}[b]{0.3\textwidth}
\includegraphics[width=\textwidth]{figures/graph.tikz}
\caption{$\graph$}
\label{fig:graph}
\end{subfigure}
\begin{subfigure}[b]{0.3\textwidth}
\includegraphics[width=\textwidth]{figures/graphx.tikz}
\caption{$\graphx$}
\label{fig:graphx}
\end{subfigure}
\begin{subfigure}[b]{0.3\textwidth}
\includegraphics[width=\textwidth]{figures/subgraphx.tikz}
\caption{$\subgraphx$}
\label{fig:subgraphx}
\end{subfigure}
\caption{The three graphs used in the interpolation procedure: an original graph $\graph$ with data sites as vertices, a slightly altered graph $\graphx$ which includes the evaluation point $x$, and a local subgraph $\subgraphx$, which in (c) consists of $x$, the shaded vertices, and bold edges. In this example, note that an edge from $\graph$ was deleted in the construction of $\graphx$.}
\label{fig:threegraphs}
\end{figure}

\section{Construction of the Interpolant}
\label{sec:interp}

Our interpolation procedure is based on graph splines, which are analogous to classical splines \cite{pesenson2009}. First we outline the method in a general graph setting, following the notation in \cite{fuselier2024}. Suppose we are working on a graph whose vertex set is a disjoint union $\vertices = \kverts\cup\uverts$, where $\kverts$ is the set of vertices at which the function values are prescribed (or ``known" vertices), and $\uverts$ is the set of interpolation locations (or ``unknown" vertices). Given function values of a function $f$ on $\kverts$, our goal is to determine the values of $f$ on $\uverts$. We denote the vector of known values as $f_\kverts$.

Graph spline interpolants are determined by minimizing an energy function given by some graph derivative matrix $\derivative$. The total energy will be measured in some norm $\norm{\cdot}$, which for our purposes is usually $\ell_2(\vertices)$ or $\wspd$. It will be convenient to reindex the columns of $\derivative$ so that
\[\derivative
=
\begin{bmatrix}
\derivative_{\uverts} & \derivative_{\kverts}
\end{bmatrix},
\]
where the submatrices consist of the columns associated with $\uverts$ and $\kverts$, respectively. With this, we can determine $f_\uverts$ (the vector of interpolated function values at the unknowns) via
\begin{equation}\label{eq:lst_squares_general}
    f_\uverts :=  \argmin_{y\in\mathbb{R}^{|\uverts|}} \norm{\derivative_{\uverts}y +\derivative_{\kverts}f_\kverts}^2.
\end{equation}
For the setting in this paper, we study a point-by-point construction in which each evaluation point is incorporated separately, i.e.  $\kverts = X$ and the unknown vertices consist of the one-point set $\uverts = \{x\}$. This brings us to the following definition.
\begin{definition}
Let $x\in \metspace$, $\derivative$ be a derivative matrix for a graph with vertices $\verticesx=X\cup \{x\}$, and $\norm{\cdot}$ be a norm on $\mathbb{R}^{|\vertices|}$. Letting $K=X$ and $U=\{x\}$, the point-by-point variational interpolant to $f$ on $X$ associated with $\derivative$ at $x$ is defined to be
\begin{equation}\label{eq:lst_squares}
\interp{\derivative}(x):=
\argmin_{y\in\mathbb{R}} \norm{\derivative_{\uverts}y +\derivative_{\kverts}f_\kverts}^2.
\end{equation}
\noindent At known vertices $x_i$, $\interp{\derivative}(x_i)$ is defined to be the function value $f(x_i)$.
\end{definition}
\noindent In what follows we will refer to this interpolant simply as ``the interpolant'' or ``the spline interpolant.''
\begin{remark}
    For ease of notation we neglected to decorate the symbol $\interp{\derivative}$ in a way that distinguishes which norm the energy is measured in. This should not cause confusion since almost all splines we consider here will use the $\ell_2$ norm, the exception being the spline associated with the random walk Laplacian, which uses $\wspd$.
\end{remark}

Using the normal equations for \eqref{eq:lst_squares} in the $\ell_2$ norm, we may write the solution as
\begin{equation}\label{eq:rational_function}
\interp{\derivative}(x)
=
\frac{-\derivative_{\uverts}^{\top}\derivative_{\kverts}f_\kverts}{\derivative_{\uverts}^{\top}\derivative_{\uverts}}.
\end{equation}
\noindent Note that the column $\derivative_\uverts$ associated with $x$ will only have nonzero entries associated with vertices that are neighbors of $x$ for lower order derivatives
(and neighbors of neighbors of $x$ if the derivative is higher order, and so on.). Thus from the products $\derivative_{\uverts}^{\top}\derivative_{\kverts}f_\kverts$ and $\derivative_U^\top \derivative_U$ in \eqref{eq:rational_function} it can be deduced that the interpolated value $\interp{\derivative}(x)$ only depends on nearest neighbors. This highlights the similarity with the $k$NN interpolant. The distinction is that the spline interpolant uses the graph structure to incorporate more of the trends in the data than $k$NN, which we will illustrate in a simple example after the following remarks.
\begin{remark}
    Note that since the distances and weights will change for different points $x$, the derivative submatrices are functions of $x$. If we work in a neighborhood of a fixed set of points $X$ and $\graphx$ has a similar structure for all $x$ considered, these matrices will have the same structure, with weights of edges containing $x$ appropriately updated. In the appendix, we use this fact to derive closed-form formulas for the spline interpolants on certain graphs.
\end{remark}

\begin{remark}
As a global interpolant, there would be non-smooth transition points where the graph structure (including the neighbors of neighbors) would change. This could be addressed, possibly using a partition of unity; see, for example, \cite{cavoretto2021partition}. However, our current focus is the local construction.
\end{remark}

\begin{example}\label{ex:stdlaprealline}
Suppose our metric space is the real line with the usual notion of distance, and we have points
$x_1<x_2<x_3<x_4$ where function values are known. Considering an arbitrary $x$ from the interval $(x_2,x_3)$, we form the weighted path graph as in \Cref{fig:graph_example_path} with distances on the graph inherited from the real line; for example, $\dist_{\graphx}(x_1,x_{2}) = x_{2}-x_1$. Using the standard Laplacian on this graph, given the data $f_1,f_2,f_3,f_4$ at the nodes, we construct the interpolant $\interp{\lap}(x)$.
A closed-form expression for this interpolant is derived
in \Cref{subsec:formula_path_subgraph_laplacian} for similar path graphs, where the interpolant is expressed in terms of data and distances. We can see from the formula \eqref{eq:pathlapinterp_appendix} that the interpolant is a rational function.
To our knowledge, this provides a new graph-based mechanism for constructing rational interpolants.
Specific examples are graphed in \Cref{fig:example_1d_standard_vs_kNN}, where we see that the weighted average definition of $k$NN (see \Cref{sec:kNN}) leads to what some might consider an underestimate, while $\interp{\lap}$ makes more of an attempt to capture the trends in the data.
\end{example}

\tikzmath{\vertexsize = .4;}

\begin{figure}[ht!]
\centering
\begin{tikzpicture}
\Vertex[label=$x$,shape=rectangle,opacity=0,size=\vertexsize]{x}

\Vertex[x=-2,y=0,label=$x_1$,opacity=0,size=\vertexsize]{x1}
\Vertex[x=-1,y=0,label=$x_2$,opacity=0,size=\vertexsize]{x2}
\Vertex[x=1,y=0,label=$x_3$,opacity=0,size=\vertexsize]{x3}
\Vertex[x=2,y=0,label=$x_4$,opacity=0,size=\vertexsize]{x4}

\Edge[](x1)(x2)
\Edge[](x2)(x)
\Edge[](x)(x3)
\Edge[](x3)(x4)
\end{tikzpicture}
\caption{\textbf{Path Graph}: Circles represent vertices where function values are known. The rectangle represents the vertex where we want to interpolate the function value.}
\label{fig:graph_example_path}
\end{figure}
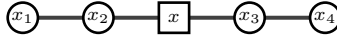

\begin{example}
In the case of where the derivative operator is the random walk Laplacian, i.e. $\derivative = \rwlap$, we use the weighted $\wspd$ norm. Thus the normal equations for \eqref{eq:lst_squares} become
\begin{equation}\label{eq:rational_function_weighted}
\interp{\rwlap}(x)
=
\frac{-\rwlap(\uverts)^{\top}D\rwlap(\kverts)f_\kverts}{\rwlap(\uverts)^{\top}D\rwlap(\uverts)} = \frac{-\lap_{\uverts}^{\top}D^{-1}\lap_{\kverts}f_\kverts}{\lap_{\uverts}^{\top}D^{-1}\lap_{\uverts}},
\end{equation}
where $D$ is the degree matrix, and $\rwlap(K)$ consists of the columns of $\rwlap$ associated with the vertices in $K$, etc. Using the same path graph in the previous example, we construct $\interp{\rwlap}$ (for a closed-form formula in this case, see \Cref{subsec:formula_path_subgraph_rwlap}). In \Cref{fig:example_1d_standard_vs_kNN}, this interpolant is plotted along with the standard Laplacian spline and the weighted $k$NN
interpolant. Notice that, like the standard Laplacian spline, the random walk spline attempts to follow trends in the data; however, the plot on the right of the figure shows that it avoids a potentially unwanted overshoot.
\end{example}

These examples hint at the different behavior of the methods depending on the smoothness and noise characteristics of the data. In particular, the spline interpolants seem particularly well suited to following smooth trends, while $k$NN offers advantages when working with noisy data. Also, the muted overshoots of the random walk spline indicate resistance to oscillatory behavior, which suggests that the random walk spline might be a nice compromise between capturing underlying trajectories and resistance to noise. We explore these points in \Cref{sec:experiments}.

\begin{figure}[ht!]
\centering
\includegraphics[width=.5\textwidth]{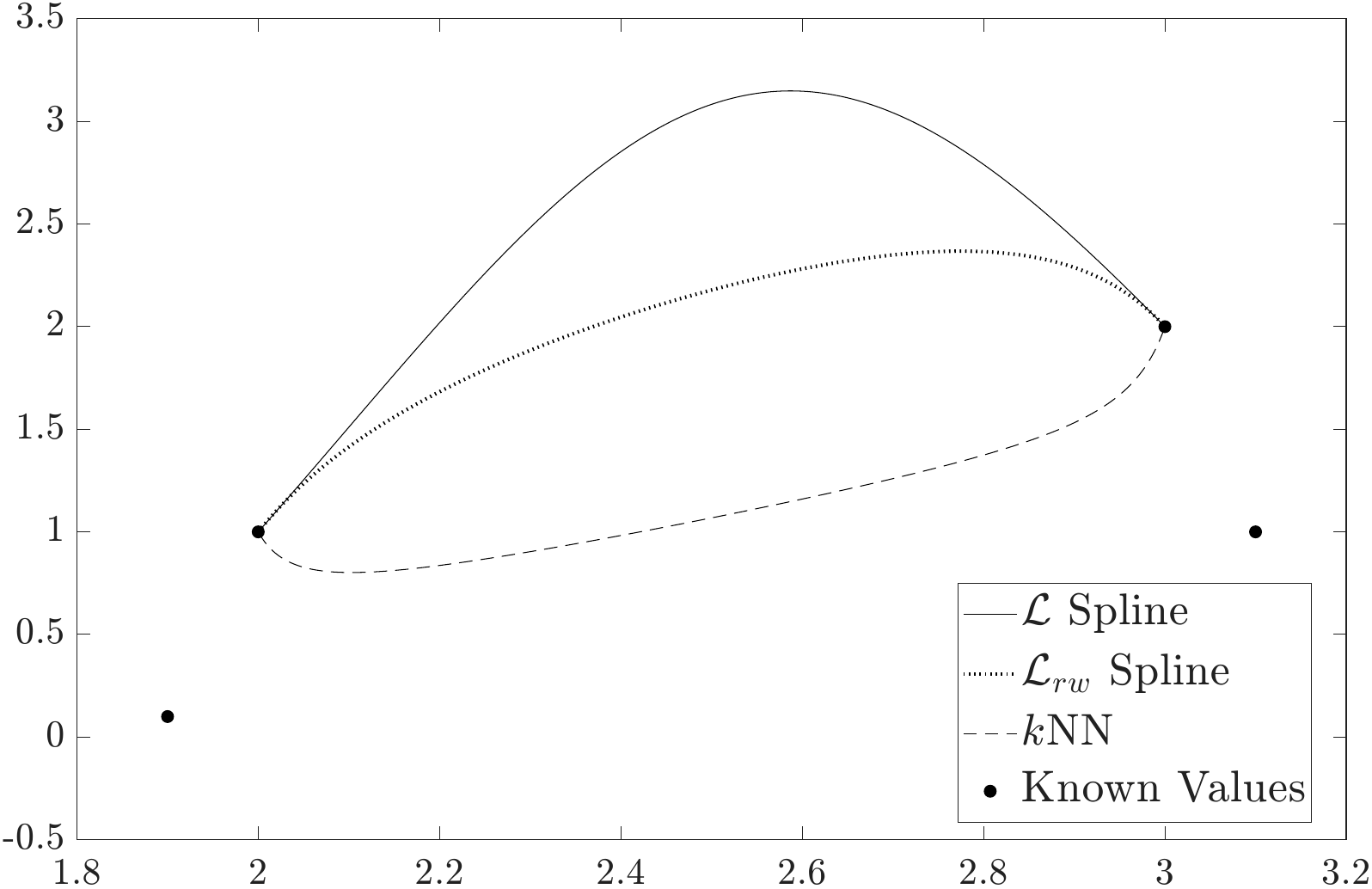}\includegraphics[width=.5\textwidth]{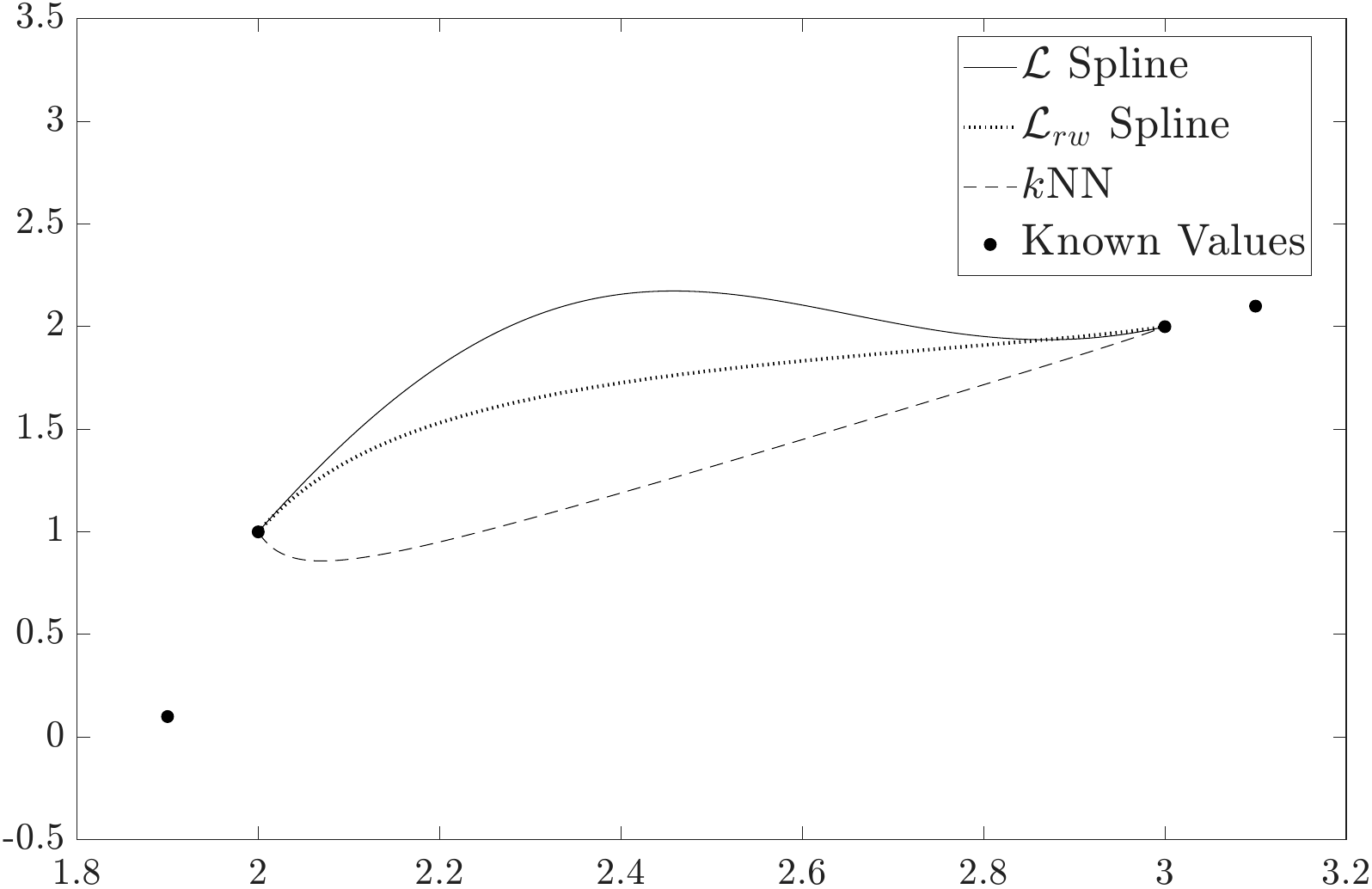}
\caption{Examples illustrating some qualitative differences between the Laplacian splines and the $k$NN interpolant on a path graph.}
\label{fig:example_1d_standard_vs_kNN}
\end{figure}

Before moving on, we reiterate that the interpolants we are considering are defined one point at a time, but one can use \eqref{eq:lst_squares_general} to interpolate at multiple points simultaneously
\cite{fuselier2024,ward2020}. However, the resulting interpolants are not the same; see \Cref{fig:example_1d_standard_vs_2point}.
In fact, if the evaluation points are equally spaced, the multi-point interpolant appears to converge to the cubic Hermite interpolant as the number of points increases. For the Hermite interpolant, the derivative information at the endpoints of the interval $[x_2,x_3]$ is determined by the slopes of the line segments connecting the two points furthest left or furthest right, respectively.

\begin{figure}[ht!]
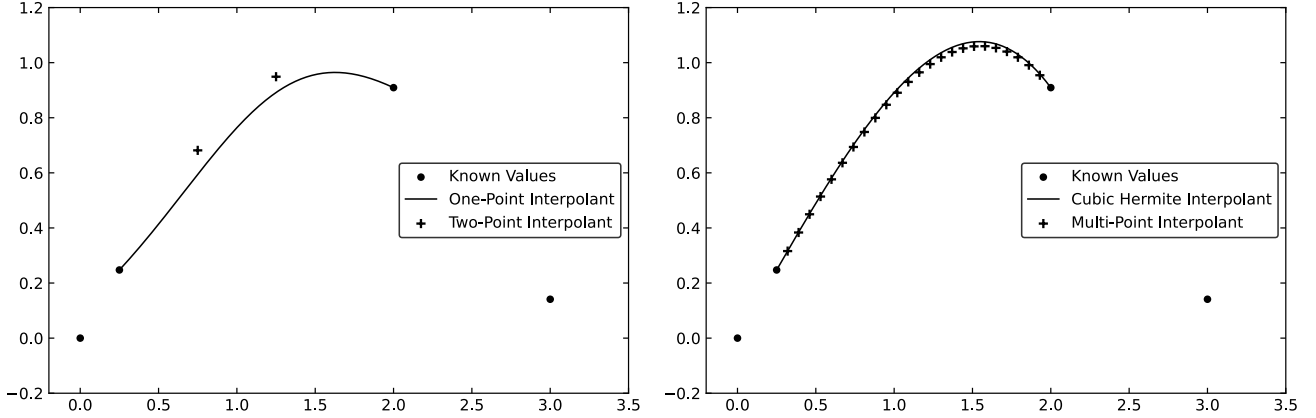

\centering
\includesvg[width=0.5\textwidth]{figures/twoptinterp1.svg}\includesvg[width=0.5\textwidth]{figures/hermite.svg}
\caption{
\textbf{Left:} The standard Laplacian interpolant on the real line (computed one point at a time) compared to the spline interpolant computed by using two evaluation points in the variational construction.
\textbf{Right:} The cubic Hermite interpolant on the real line compared to the multi-point spline interpolant computed by using 24 evaluation points in the variational construction.
}
\label{fig:example_1d_standard_vs_2point}
\end{figure}

\section{Properties of the Interpolants}\label{sec:properties}

In this section, we investigate several properties of the interpolation procedures introduced in \Cref{sec:interp} together with linear interpolants, which we define below.
\begin{definition}
For points $x_1,x_2\in\metspace$ with data values $f_1,f_2$, we define the \emph{linear interpolant} via
\begin{equation}\nonumber
\interplinear(x)
=
\frac{f_1\dist(x,x_2)
+
f_2\dist(x_1,x)}{\dist(x_1,x) +\dist(x,x_2)},
\end{equation}
where $x\in \metspace$.\footnote{Although the formula defines $\interplinear$ on all of $\metspace$, it should not be interpreted as a globally linear function. However, it is linear with respect to arclength when restricted to a geodesic joining $x_1$ and $x_2$.}
\end{definition}
\noindent The section begins by establishing an energy estimate for the standard Laplacian interpolant on edges and briefly justify the choice of edge weights. We then derive estimates for the Lipschitz constants of the $k$NN, random walk, and standard graph spline interpolants on path subgraphs. Finally, we establish stability estimates for the interpolation procedures on general graphs.

\subsection{Interpolation on Edges}

Consider an edge $(x_1,x_2)$ in a graph $\graph$, and let $x$ be a point on
a geodesic in the metric space between the pair.
Now construct the graph $\graphx$ by removing that edge and inserting edges $(x_1,x)$ and $(x,x_2)$.
We wish to show that the Laplacian interpolant $\interp{\lap}$ on $\graph_x$ is at least as smooth as the original function $f$ on $\graph$.

Our arguments will use linear interpolants. Consider the simple path subgraph $\subgraphx$ consisting of two data sites $x_1$, $x_2$, each connected to the unknown site $x$ (see \Cref{fig:graph_example_shortpath}). Below we see that, for this subgraph, we can view the linear interpolant $\interplinear$ as $\interp{\fodlapfactor}$, where $\fodlapfactor$ is the first order differential operator $\fodlapfactor=W^{1/2}M$.

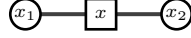
\begin{figure}[ht!]
\centering
\begin{tikzpicture}
\Vertex[label=$x$,shape=rectangle,opacity=0,size=\vertexsize]{x}

\Vertex[x=-1,y=0,label=$x_1$,opacity=0,size=\vertexsize]{x1}
\Vertex[x=1,y=0,label=$x_2$,opacity=0,size=\vertexsize]{x2}

\Edge[](x1)(x)
\Edge[](x)(x2)
\end{tikzpicture}
\caption{\textbf{Short Path Graph}: Circles represent vertices where function values are known. The rectangle represents the vertex where we want to interpolate the function value.}
\label{fig:graph_example_shortpath}
\end{figure}

\begin{proposition}\label{prop:WM_linear_interpolant}
Consider the subgraph $\subgraphx$ of \Cref{fig:graph_example_shortpath},
with data values $f_1,f_2$ at vertices $x_1,x_2$.
The interpolant $\interp{\fodlapfactor}$ is the linear interpolant $\interplinear$.
\end{proposition}
\begin{proof}
The interpolant is formed by solving \eqref{eq:lst_squares}, which in this case involves minimizing the expression
\begin{align*}
\abs{(y-f_1)\weight{x}{x_1}^{1/2}}^2
+
\abs{(f_2-y)\weight{x}{x_2}^{1/2}}^2
&=
\abs{\frac{y-f_1}{\dist(x,x_1)^{1/2}}}^2
+
\abs{\frac{f_2-y}{\dist(x,x_2)^{1/2}}}^2\\
&=
\frac{y^2-2yf_1+f_1^2}{\dist(x,x_1)}
+
\frac{y^2-2yf_2+f_2^2}{\dist(x,x_2)}.
\end{align*}
This occurs for
\begin{equation}\nonumber
y
=
\frac{f_1\dist(x,x_2)+f_2\dist(x,x_1)}{\dist(x_1,x)+\dist(x,x_2)},
\end{equation}
so $\interp{\fodlapfactor}(x)$ is the linear interpolant.
\end{proof}

\begin{remark}
The fact that $\interp{\fodlapfactor}$ creates a linear interpolant relies on and provides justification for the choice
$\weight{x}{y} = \dist(x,y)^{-1}$. Further justification, based on how the standard Laplacian interpolant $\interp{\lap}$ preserves harmonic functions, is in \Cref{subsec:weightanddistances}.
\end{remark}

\begin{remark}
The linear interpolant is also an $\ell_1$ and $\ell_{\infty}$ spline (suitably defined) with respect to the derivative $\fodgrad=W M$. However, in general, there is not a unique spline solution in these norms. See \cite{fisher1975spline,unser2017splines} for continuous-domain discussion of this point.
\end{remark}

The next result shows that the linear interpolant preserves energy. To avoid cumbersome notation, we use the same symbol $\lap$ for the graph Laplacian on both $\graph$ and the augmented graph $\graphx$.\footnote{Although these are technically different operators, their domains will always be clear from the function to which they are applied.}

\begin{lemma}\label{lem:splitEdge_equalLaplacian}
Let $X=\set{x_i}_{i=1}^n\subset \metspace$ and suppose $x$ is on a shortest path in $\metspace$ from $x_1$ to $x_2$. The linear interpolant
$\interplinear$ on $\graphx$
to a function
$f$ on $\graph$
satisfies
\begin{align*}
\begin{split}
\lap \interplinear (x_i)
&= \lap f(x_i),
\quad \text{for } 1\leq i \leq n\\
\lap \interplinear (x)
&= 0.
\end{split}
\end{align*}
\end{lemma}
\begin{proof}
Since $x$ is on a shortest path, we have
$\dist_{\graph}(x_1,x_2)=\dist_{\graphx}(x_1,x)+\dist_{\graphx}(x,x_2)$.
Using this and the definition of the linear interpolant, observe that
\begin{equation}\label{eq:xedgediff}
  \frac{\interplinear(x_1)-\interplinear(x)}{\dist_{\graphx}(x_1,x)}
  =
  \frac{f(x_1) - f(x_2)}{\dist_{\graph}(x_1,x_2)}, \quad\mbox{and}\quad
  \frac{\interplinear(x_2)-\interplinear(x)}{\dist_{\graphx}(x_1,x)} = \frac{f(x_2)-f(x_1)}{\dist_{\graph}(x_1,x_2)}.
\end{equation}
Thus, considering the Laplacian of the interpolant at $x$, we get
\begin{equation}\nonumber
\lap \interplinear(x)
=
\frac{\interplinear(x)-\interplinear(x_1)}{\dist_{\graphx}(x_1,x)}
+
\frac{\interplinear(x)-\interplinear(x_2)}{\dist_{\graphx}(x,x_2)}
=
\frac{f(x_2) -f(x_1)}{\dist_{\graph}(x_1,x_2)}
+
\frac{f(x_1)-f(x_2)}{\dist_{\graph}(x_1,x_2)}
=
0,
\end{equation}
demonstrating that the interpolant is harmonic at $x$.

Using \eqref{eq:xedgediff} and the expanded form of the Laplacian in \eqref{eq:lapf_expanded}, we get
\begin{align*}
\begin{split}
\lap \interplinear(x_1)
&=
\frac{\interplinear(x_1)-\interplinear(x)}{\dist_{\graphx}(x_1,x)}
+
\sum_{x_1\sim x_i, x_i\neq x_2}
\frac{\interplinear(x_1)-\interplinear(x_i)}{\dist_{\graphx}(x_1,x_i)} \\
&=
\frac{f(x_1)-f(x_2)}{\dist_{\graph}(x_1,x_2)}
+
\sum_{x_1\sim x_i, x_i\neq x_2}
\frac{f(x_1)-f(x_i)}{\dist_{\graph}(x_1,x_i)} = \lap f(x_1)
\end{split}
\end{align*}
Similarly, $\lap \interplinear(x_2)=\lap f(x_2)$.
Note that $\lap  \interplinear(x_i) = \lap f(x_i)$ for the other values of $i$ because they are not adjacent to $x$.
\end{proof}

This brings us to the following energy estimate for the Laplacian interpolant $\interp{\lap}$.
\begin{theorem}\label{thm:energybound}
If $x$ is chosen on a shortest path in $\metspace$ from $x_1$ to $x_2$, then the interpolant $\interp{\lap}$ to the function $f$ satisfies
\begin{equation}\nonumber
\norm{\lap \interp{\lap}}_2
\leq
\norm{\lap f}_2
\end{equation}
\end{theorem}
\begin{proof}
By \eqref{eq:lst_squares}, $\interp{\lap}$ is the interpolant with minimal energy.
Since
$\interplinear$ is an interpolant, we have
\begin{equation}\nonumber
\norm{\lap \interp{\lap} }_2
\leq
\norm{\lap \interplinear}_2,
\end{equation}
and \Cref{lem:splitEdge_equalLaplacian} implies
\begin{equation}\nonumber
\norm{\lap \interplinear}_2
=
\norm{\lap f}_2.
\end{equation}
\end{proof}
This result says that the Laplacian interpolant is at least as smooth on $\graphx$ as the known data on $\graph$.
In general, $\interp{\lap}$ will not be a linear interpolant, so the semi-norm of the Laplacian interpolant can
be strictly smaller than the semi-norm of the original data.
\begin{corollary}\label{cor:harmonic_preservation}
If $f$ is harmonic, then $\interp{\lap}$ is harmonic along geodesics connecting edge endpoints. Moreover, if $f$ is harmonic at the end vertices $x_1$ and $x_2$, then $\interp{\lap}$ is harmonic at $x_1,x_2, x$.
\end{corollary}
\begin{proof}
Any energy of $\lap f$ is fully contained at the vertices other than $x_1,x_2$, and note that these values will be the same for $\lap \interp{\lap}$.
Since $\norm{\lap \interp{\lap}}_2 \leq \norm{\lap f}_2 $
we must have $\lap \interp{\lap}$ being zero at $x_1,x_2, x$.
\end{proof}

\subsection{Weights and distances}\label{subsec:weightanddistances}

In \Cref{prop:WM_linear_interpolant} we saw that $\interp{\fodlapfactor}$ is equivalent to the linear interpolant. A key underlying assumption was the relationship between weights and distances, specifically weights on edges are defined as the inverse of the metric space distance between the vertices defining the edge. Here, we consider the converse question:
\emph{If the interpolant for the first-order derivative $\fodlapfactor$ is the linear interpolant, what is the necessary relationship between weights (defining $\interp{\fodlapfactor}$) and distances (defining $\interplinear$)?}
We then ask an analogous question for the standard Laplacian:
\emph{For what relationship between weights and distances is $\interp{\lap}$ harmonic whenever $f$ is?}

To make our argument, we consider a first-order derivative $\fodlapfactor = W^{1/2} M$ with general weight matrix $W$ and the condition that $\interp{\fodlapfactor}$ gives the linear interpolant. We require this equivalence to hold on all valid metric spaces, short path graphs $\subgraphx$,  and points $x$. Our first result identifies a general constraint on the relationship between the weights and the metric distances.

\begin{lemma}\label{lem:weight_distance_relate_Dr}
If $\interp{\fodlapfactor}$ is always the linear interpolant on short path graphs $\subgraphx$, then
\begin{align*}
\frac{\dist(x,x_2)}{\dist(x_1,x) +\dist(x,x_2)}
&=
\frac{\weight{x_1}{x}}{\weight{x_1}{x}+ \weight{x}{x_2}}\\
\frac{\dist(x,x_1)}{\dist(x_1,x) +\dist(x,x_2)}
&=
\frac{\weight{x_2}{x}}{\weight{x_1}{x}+ \weight{x}{x_2}}
\end{align*}
\end{lemma}
\begin{proof}
The argument follows the proof of \Cref{prop:WM_linear_interpolant}.
The interpolant $\interp{\fodlapfactor}$ is formed by solving the least-squares problem, which in this case involves minimizing the expression
\begin{equation}\nonumber
\abs{(y-f_1)\weight{x_1}{x}^{1/2}}^2
+
\abs{(f_2-y)\weight{x}{x_2}^{1/2}}^2
=
\weight{x_1}{x} (y^2-2yf_1+f_1^2)
+
\weight{x}{x_2} (y^2-2yf_2+f_2^2)
\end{equation}
over $y\in\mathbb{R}$. The derivative is
\begin{equation}\nonumber
\weight{x_1}{x} (2y-2f_1)
+
\weight{x}{x_2} (2y-2f_2)
=
2y(\weight{x_1}{x}+ \weight{x}{x_2} )
-
2( \weight{x_1}{x} f_1 + \weight{x}{x_2} f_2),
\end{equation}
which is 0 for
\begin{equation}\nonumber
y
=
\frac{\weight{x_1}{x} f_1 + \weight{x}{x_2} f_2}{\weight{x_1}{x}+ \weight{x}{x_2} }.
\end{equation}
Comparing the coefficients of $f_j$ with the coefficients in $\interplinear$, we get the result.
\end{proof}

\begin{theorem}\label{thm:weights_dist_linearrep}
Assume the hypotheses of \Cref{lem:weight_distance_relate_Dr}.
If weights are a function of distance, i.e.
$\weight{x}{y}=\varphi(\dist(x,y))$ for some
$\varphi:(0,\infty)\to(0,\infty)$, then
\begin{equation}\nonumber
\weight{x}{y}
=
C\dist(x,y)^{-1}
\end{equation}
for some constant $C> 0$.
\end{theorem}

\begin{proof}
Rearranging the first equation from \Cref{lem:weight_distance_relate_Dr},
\begin{align*}
\begin{split}
\dist(x,x_2)
(\weight{x_1}{x}+ \weight{x}{x_2})
&=
\weight{x_1}{x}
\parenth{\dist(x_1,x) + \dist(x,x_2) }\\
\dist(x,x_2)\weight{x}{x_2}
&=
\weight{x_1}{x}\dist(x_1,x).
\end{split}
\end{align*}
Using the required form for the weights,
\begin{equation}\nonumber
\dist(x,x_2)\varphi(\dist(x,x_2))
=
\varphi(\dist(x_1,x))\dist(x_1,x).
\end{equation}
Since we require this relationship to be valid for all possible graphs and metric spaces,
given two arbitrary numbers $t,T>0$, we can construct a graph in Euclidean space with
$\dist(x_1,x) = t$ and $\dist(x,x_2)=T$, leading to
\begin{equation}\nonumber
T\varphi(T)
=
\varphi(t)t.
\end{equation}
Therefore, $\varphi(t)t$ must be constant, a positive constant $C$.
\end{proof}

There is another interesting relationship that can be derived by assuming a harmonic target function implies a harmonic interpolant.

\begin{lemma}\label{lem:weights_inverse_sum}
   Assume that when
   $f$ on $\graph$ is harmonic at endpoints $x_1,x_2$ of an edge that we then have $\interp{\lap}$ on $\graphx$ being harmonic at $x_1,x,x_2$. Then we get the identity
\begin{equation}\label{eq:weights_inverse_sum}
\frac{1}{\weight{x_1}{x_2}}
=
\frac{1}{\weight{x_1}{x}}
+
\frac{1}{\weight{x_2}{x}}.
\end{equation}
\end{lemma}
\begin{proof}
The key observation is the structure of the vertex-based form  of the Laplacian \eqref{eq:lapf_expanded}.
If $\lap f(x_1)= \lap f(x_2)=0$, to keep the interpolant harmonic at $x_1$, we must have
\begin{equation}\nonumber
(f(x_1)-f(x_2))\weight{x_1}{x_2}
=
(\interp{\lap}(x_1)-\interp{\lap}(x))\weight{x_1}{x}
\end{equation}
since all other terms defining $\lap\interp{\lap}(x_1)$ match the corresponding terms defining $\lap f$ at $x_1$. Similarly,
\begin{equation}\nonumber
(f(x_2)-f(x_1))\weight{x_1}{x_2}
=
(\interp{\lap}(x_2)-\interp{\lap}(x))\weight{x_2}{x}.
\end{equation}
Finally, to be harmonic at $x$,
\begin{equation}\nonumber
(\interp{\lap}(x)-\interp{\lap}(x_1))\weight{x}{x_1}
=
-
(\interp{\lap}(x)-\interp{\lap}(x_2))\weight{x}{x_2}.
\end{equation}
Using the above equations, we get
\begin{align*}
\frac{\interp{\lap}(x_1)-\interp{\lap}(x)}{\weight{x_1}{x_2}}
-
\frac{\interp{\lap}(x_2)-\interp{\lap}(x)}{\weight{x_1}{x_2}}
&=
\frac{f(x_1)-f(x_2)}{\weight{x_1}{x}}
-
\frac{f(x_2)-f(x_1)}{\weight{x_2}{x}},\\
\end{align*}
which results in
\begin{align*}
\parenth{\interp{\lap}(x_1)-\interp{\lap}(x_2)}\frac{1}{\weight{x_1}{x_2}}
&=
\parenth{f(x_1)-f(x_2)}
\parenth{
\frac{1}{\weight{x_1}{x}}
+
\frac{1}{\weight{x_2}{x}}
}.
\end{align*}
In the situation where $f(x_1)\neq f(x_2)$, we obtain the result.
\end{proof}

\begin{remark}
The condition \eqref{eq:weights_inverse_sum} is precisely the \emph{series resistance identity} from electrical network theory, expressed in terms of the edge weights.
Interpreting the reciprocals of the edge weights as resistances, this identity states that subdividing an edge by inserting an intermediate vertex preserves the effective resistance between $x_1$ and $x_2$.
\end{remark}

Now we use \eqref{eq:weights_inverse_sum} to show, as in  \Cref{thm:weights_dist_linearrep}, that under these assumptions the edge weights are necessarily inversely proportional to the metric distance.

\begin{theorem}
Assume the hypotheses of \Cref{lem:weights_inverse_sum} hold. If the weights are a continuous function of distance, i.e.
$\weight{x}{y}=\varphi(\dist(x,y))$ for some
$\varphi:(0,\infty)\to(0,\infty)$, then
\begin{equation}\nonumber
\weight{x}{y}
=
C\dist(x,y)^{-1}
\end{equation}
for some constant $C> 0$. For these weights, harmonic reproduction at an inserted point occurs whenever the inserted point lies on a geodesic joining the endpoints of the edge.
\end{theorem}
\begin{proof}
The proof exploits the additive structure of distances along geodesics. Let $m,n$ be positive integers, and let $x_0,x_1,\dots,x_n$ be a sequence of equally spaced points on a geodesic between vertices $x_0,x_n$ that are connected by an edge in $\graph$. If $\dist(x_0,x_n)=m$, then repeatedly applying \eqref{eq:weights_inverse_sum} gives
\begin{align*}
\frac{1}{\varphi(m)}
&=
\sum_{i=1}^n \frac{1}{\varphi(m/n)}.
\end{align*}
Rearranging gives
\begin{equation}\nonumber
\frac{m}{n}
\varphi \parenth{\frac{m}{n} }
=
m \varphi(m).
\end{equation}
Since $m,n$ are arbitrary, we can select values to show that $t\varphi(t)$ is constant on the positive rationals. By continuity, $t\varphi(t)$ is a fixed constant $C>0$ for all $t>0$.

Finally, we determine when these weights are compatible with harmonic reproduction. Combining $\varphi(t) = C/t$ with \eqref{lem:weights_inverse_sum} implies that
\[\dist(x_1,x_2) = \dist(x,x_1) + \dist(x,x_2),\]
which holds precisely when $x$ is on a geodesic between $x_1$ and $x_2$.
\end{proof}

\subsection{Lipschitz Constants on Path Subgraphs}

Recall that on any metric space $(Y,\dist_Y)$ (including a graph) a function $f:Y\rightarrow \mathbb{R}$ is Lipschitz continuous if there is a constant $C>0$ such that
\begin{equation}\nonumber
\abs{f(x)-f(y)}
 \leq
 C
 \dist_Y(x,y)
\end{equation}
for all $x,y$. The smallest constant is called the Lipschitz constant, which we denote by $\lip{f}$.

We begin with two elementary observations concerning Lipschitz constants on graphs and their ambient metric spaces.
\begin{proposition}
For a function $f:\graph\rightarrow \mathbb{R}$,
$\norm{\fodgrad f}_{\infty} =\norm{WM f}_{\infty}$ is the Lipschitz constant $\lip{f}$.
\end{proposition}

\begin{proposition}
Suppose a graph is embedded in a metric space, and the distance along each graph edge is no greater than the corresponding distance in the metric space. Then, for any Lipschitz function defined on the embedding space, the Lipschitz constant of its restriction to the graph is bounded by its Lipschitz constant with respect to the ambient space.
\end{proposition}

We now turn to estimating graph Lipschitz constants for the interpolants $\interp{\lap}$, $\interpkNN$, and $\interp{\rwlap}$ on simple path subgraphs. Consider any path $x_1,x_2,x_3,x_4$ in a graph $\graph$. One can construct a new graph $\graph_x$ by removing the edge $(x_2,x_3)$ and inserting edges $(x_2,x)$ and $(x,x_3)$, where $x\in\metspace$. The interpolant's value at $x$ will be determined by spline interpolation on the path subgraph $\subgraphx$ (see \Cref{fig:graph_example_path}, for example). We begin with the standard Laplacian spline.

\begin{theorem}\label{thm:lipschitz_lap_G_Gx}
Let $f$ be a function on $\graph$, $x\in\metspace$, and consider $\graphx$ as described above. The Lipschitz constant for the interpolant $\interp{\lap}$ on $\graphx$ is bounded by a multiple of the Lipschitz constant for $f$ on $\graph$:
\begin{equation}\nonumber
\lip{\interp{\lap}}
\leq
\frac{4}{3}\lip{f}.
\end{equation}
\end{theorem}
\begin{proof}
Our proof is direct, using a formula for $\interp{\lap}(x)$ derived in \Cref{subsec:formula_path_subgraph_laplacian}. For ease of notation, we denote distances on $\subgraphx$ as
$\dist_1=\dist_{\subgraphx}(x_1,x_2)$,
$\dist_2=\dist_{\subgraphx}(x_2,x)$,
$\dist_3=\dist_{\subgraphx}(x_3,x)$,
$\dist_4=\dist_{\subgraphx}(x_3,x_4)$.
Referring to the formula \eqref{eq:pathlapinterp_appendix} in \Cref{subsec:formula_path_subgraph_laplacian} and collecting the given function values,
the standard Laplacian spline is given by
\begin{equation}\label{eq:stdpathinterp}
   \interp{\lap}(x)
   = \frac{1}{A}\left(
   -\frac{\dist_2 \dist_3^2}{\dist_1}f_1
   +\left(\frac{\dist_2\dist_3^2}{\dist_1} + (\dist_2 + 2\dist_3)\dist_3\right)f_2
   +\left(\frac{\dist_3\dist_2^2}{\dist_4} + (2\dist_2 + \dist_3)\dist_2\right)f_3
   -\frac{\dist_3\dist_2^2}{\dist_4}f_4\right),
\end{equation}
where
\begin{equation}\nonumber
A
=
\dist_2^2 +(\dist_2+\dist_3)^2 +\dist_3^2
=
2\dist_2^2 +2\dist_2\dist_3 +2\dist_3^2.
\end{equation}
Rearranging gives
\begin{equation}\label{eq:stdlap_dform}
   \interp{\lap}(x)
   = \frac{1}{A}\left(
   \frac{\dist_2 \dist_3^2}{\dist_1}(f_2 - f_1)
   +(\dist_2 + 2\dist_3)\dist_3 f_2
   +(2\dist_2 + \dist_3)\dist_2 f_3
   +\frac{\dist_3\dist_2^2}{\dist_4}(f_3 - f_4)\right).
\end{equation}
Thus we get
\begin{align}\begin{split}\nonumber
    \frac{1}{\dist_2}(\interp{\lap}(x) - f_2)
    &= \frac{1}{A}\left(
   \frac{\dist_3^2}{\dist_1}(f_2 - f_1)
   +(\dist_2 + 2\dist_3)\frac{\dist_3}{\dist_2} f_2
   +(2\dist_2 + \dist_3) f_3
   +\frac{\dist_3\dist_2}{\dist_4}(f_3 - f_4)\right)-\frac{1}{A}\frac{A f_2}{\dist_2}\\
    &= \frac{1}{A}\left(
   \frac{\dist_3^2}{\dist_1}(f_2 - f_1)
   +\frac{1}{\dist_2}\left((\dist_2 + 2\dist_3)\dist_3-A\right)f_2
   +(2\dist_2 + \dist_3) f_3
   +\frac{\dist_3\dist_2}{\dist_4}(f_3 - f_4)\right).
\end{split}\end{align}
Note that the coefficient of $f_2$ simplifies to
\begin{align}\begin{split}\nonumber
    \frac{1}{\dist_2}\left((\dist_2 + 2\dist_3)\dist_3-A\right)
    & =
    \frac{1}{\dist_2}\left(\dist_2 \dist_3 +2\dist_3^2 - 2\dist_2^2 - 2\dist_2\dist_3 - 2\dist_3^2\right)=
    \frac{1}{\dist_2}\left(-\dist_2 \dist_3 - 2\dist_2^2\right)= -\dist_3 - 2\dist_2,
\end{split}\end{align}
which yields
\begin{align}\begin{split}\label{eq:stdlappath_est}
    \frac{1}{\dist_2}\left|\interp{\lap}(x) - f_2\right|
    &=
    \frac{1}{A}\left|
   \frac{\dist_3^2}{\dist_1}(f_2 - f_1)
   +(2\dist_2 + \dist_3) (f_3 - f_2)
   +\frac{\dist_3\dist_2}{\dist_4}(f_3 - f_4)\right| \\
   &=
   \frac{1}{A}\left|
   \dist_3^2\frac{f_2 - f_1}{\dist_1}
   +(2\dist_2 + \dist_3)(\dist_2+\dist_3)\frac{f_3 - f_2}{\dist_2+\dist_3}
   +\dist_3\dist_2\frac{f_3 - f_4}{\dist_4}\right|\\
    & \leq
    \lip{f}
    \frac{1}{A}\left(\dist_3^2 +(2\dist_2+\dist_3)(\dist_2+\dist_3) +\dist_2\dist_3\right) \\
    &= \lip{f}
    \frac{1}{A}\left(\dist_3^2 + 2\dist_2^2 +3\dist_2\dist_3 +\dist_3^2 +\dist_2\dist_3\right) = \lip{f}
    \frac{1}{A}\left(2\dist_3^2 + 2\dist_2^2 +2\dist_2\dist_3 + 2\dist_2\dist_3\right) \\
    &= \lip{f}
    \frac{1}{A}\left(A + 2\dist_2\dist_3\right) = \lip{f}
    \left(1 +\frac{2\dist_2\dist_3}{A}\right),
\end{split}\end{align}
where we have used the fact that
\[\frac{|f_3-f_2|}{\dist_2+\dist_3} \leq \frac{|f_3-f_2|}{\dist_{\graph}(x_2,x_3)} \leq \lip{f}.\]
Finally, let
$\dist_2 = \alpha (\dist_2 +\dist_3)$, which implies that $0\leq \alpha \leq 1$ and
$\dist_3 = (1-\alpha)(\dist_2 +\dist_3)$,
and note that
\begin{equation}\nonumber
\frac{2\dist_2 \dist_3}{A}
=
\frac{\dist_2\dist_3}{\dist_2^2 +\dist_3^2+ \dist_2 \dist_3}
=
\frac{\alpha(1-\alpha)}{\alpha^2 +(1-\alpha)^2 +\alpha(1-\alpha)}
=
\frac{\alpha(1-\alpha)}{\alpha^2 +1-\alpha}.
\end{equation}
This final expression is maximized over $[0,1]$ when $\alpha = 1/2$, which gives
\begin{equation}\nonumber
\frac{2\dist_2\dist_3}{A} \leq \frac{1}{3}.
\end{equation}
Applying this to \eqref{eq:stdlappath_est}, it follows that
\begin{equation}\nonumber
\frac{1}{\dist_2}\left|\interp{\lap}(x) - f_2\right|
\leq
\frac{4}{3}\lip{f}.
\end{equation}
Similarly,
\begin{equation}\nonumber
\frac{1}{\dist_3}\left|\interp{\lap}(x) - f_3\right| \leq \frac{4}{3}\lip{f},
\end{equation}
which completes the proof.
\end{proof}

\begin{remark}
Inspection of \eqref{eq:stdlap_dform} shows how the standard Laplacian spline depends on the slope endpoint data through the quantities $(f_2-f_1)/\dist_1$, $f_2$, $f_3$, ($f_3-f_4)/\dist_4$. If these quantities are held fixed, then $\interp{\lap}(x)$ is unchanged.
\end{remark}

\begin{example}\label{example:stdlip}
To see that \Cref{thm:lipschitz_lap_G_Gx} is sharp, consider the following example on $\metspace=\mathbb{R}$. let $x_1=1$, $x_2=2$, $x_3=3$, $x_4=4$, $x=2.5$, and let $f_1=0$, $f_2=1$, $f_3=2$, $f_4=1$. This gives $\lip{f}=1$. Using \eqref{eq:stdpathinterp}, we get $\interp{\lap}(x)=5/3$ and $|f_2 - \interp{\lap}(x)|/\dist_2 = 4/3$. For a plot of the interpolant over the entire interval $(x_2,x_3)$, see \Cref{subfig:stdlipexample}.
\end{example}

Now we investigate similar bounds when using the $k$NN interpolant or the random walk spline. The following results show that both can allow more slope magnification than the standard Laplacian spline.

\begin{theorem}\label{thm:lipschitz_knn_Gx_G}
Let $f$ be a function on $\graph$, $x\in\metspace$, and consider $\graphx$ as described above. The Lipschitz constant for the interpolant $\interpkNN$ on $\graphx$ is bounded by a multiple of the Lipschitz constant for $f$ on $\graph$:
\begin{equation}\nonumber
\lip{\interpkNN}
\leq
3\lip{f}.
\end{equation}
\end{theorem}

\begin{proof}
    Referring to the formula \eqref{eq:kNNpathformula} in \Cref{sec:kNN} and collecting the given function values,
the $k$NN interpolant is given by
\begin{equation}\label{eq:kNNpathinterp}
   \interpkNN(x)
   =
   \frac{1}{A}\left(
   \frac{\dist_2 \dist_3}{\dist_1 + \dist_2}f_1
   +\dist_3 f_2
   +\dist_2 f_3
   +\frac{\dist_3\dist_2}{\dist_3 + \dist_4}f_4\right),
\end{equation}
where we have defined $\rho_i$ as before and
\begin{equation}\nonumber
A
=
\frac{\dist_2 \dist_3}{\dist_1 + \dist_2}
   +\dist_3
   +\dist_2
   +\frac{\dist_3\dist_2}{\dist_3 + \dist_4}.
\end{equation}
Thus we get
\begin{align*}
\interpkNN(x) - f_2 & = \frac{1}{A}\left(
   \frac{\dist_2 \dist_3}{\dist_1 + \dist_2}f_1
   +\dist_3 f_2
   +\dist_2 f_3
   +\frac{\dist_3\dist_2}{\dist_3 + \dist_4}f_4 - Af_2\right)\\
   & = \frac{1}{A}\left(
   \frac{\dist_2 \dist_3}{\dist_1 + \dist_2}(f_1-f_2)
   +\dist_2 (f_3 - f_2)
   +\frac{\dist_3\dist_2}{\dist_3 + \dist_4}(f_4 - f_2)\right),
\end{align*}
which gives
\begin{equation}\label{eq:knnlipexpr}
    \frac{\interpkNN(x) - f_2}{\dist_2} =
   \frac{\dist_3}{\dist_1 + \dist_2}\frac{(f_1-f_2)}{A}
   +\frac{(f_3 - f_2)}{A}
   +\frac{\dist_3}{\dist_3 + \dist_4}\frac{(f_4 - f_2)}{A}.
\end{equation}
Note that $\dist_3/A \leq 1$, so the first term is bounded by $\lip{f}$:
\begin{equation}\nonumber
    \left |\frac{\dist_3}{\dist_1 + \dist_2}\frac{(f_1-f_2)}{A}\right| \leq \frac{|f_1-f_2|}{\dist_1 + \dist_2} <  \frac{|f_1-f_2|}{\dist_1} \leq \lip{f}.
\end{equation}
Also, since $A >  \dist_2+\dist_3$, for the second term we have
\begin{equation}\nonumber
    \left|\frac{(f_3 - f_2)}{A}\right| < \frac{|f_3-f_2|}{\dist_2+\dist_3} \leq \lip{f}.
\end{equation}
 For the last term, note that $A>\dist_3 + \frac{\dist_2\dist_3}{\dist_3+\dist_4}$, so we get
\begin{equation}\nonumber
    \left|\frac{\dist_3}{\dist_3 + \dist_4}\frac{(f_4 - f_2)}{A}\right| <\frac{\dist_3}{\dist_3 + \dist_4}\frac{|f_4 - f_2|}{\left(\dist_3 + \frac{\dist_2\dist_3}{\dist_3+\dist_4}\right)} =  \frac{|f_4 - f_2|}{\dist_2+\dist_3+\dist_4} \leq \lip{f}.
\end{equation}
The result follows.
\end{proof}

\begin{example}\label{example:knnlip}
To see that \Cref{thm:lipschitz_knn_Gx_G} is sharp, consider the example with $x_1=1$, $x_2=2$, $x_3=3$, $x_4=4$ with $f_1=1$, $f_2=0$, $f_3=1$, $f_4=2$, giving $\lip{f}=1$. Let $x = x_2+\epsilon$, and note that as $\epsilon\rightarrow 0$ we have
\[\dist_2\rightarrow 0,\quad A\rightarrow \dist_3\rightarrow \dist(x_2,x_3).\]
Applying these to \eqref{eq:knnlipexpr}, we get
\[\frac{\interpkNN(x) - f_2}{\dist_2} \rightarrow 3.\]
\end{example}

Lastly, we consider the random walk interpolant.

\begin{theorem}\label{thm:lipschitz_rwlap_Gx_G}
Let $f$ be a function on $\graph$, $x\in\metspace$, and consider $\graphx$ as described above. The Lipschitz constant for the interpolant $\interp{\rwlap}$ on $\graphx$ is bounded by a multiple of the Lipschitz constant for $f$ on $\graph$:
\begin{equation}\nonumber
\lip{\interp{\rwlap}}
\leq
3\lip{f}.
\end{equation}
\end{theorem}
\begin{proof}
Referring to the formula in \Cref{subsec:formula_path_subgraph_laplacian},
the random walk spline is given by
\begin{equation}\nonumber
   \interp{\rwlap}(x)
   = \frac{1}{A}\left(
   -\frac{\dist_2 \dist_3}{\dist_1+\dist_2}f_1
   +2\dist_3f_2
   +2\dist_2f_3
   -\frac{\dist_2 \dist_3}{\dist_3+\dist_4}f_4\right),
\end{equation}
where
\begin{equation}\label{eq:Arw}
A = \frac{\dist_1 \dist_3}{\dist_1+\dist_2} +(\dist_2+\dist_3) +\frac{\dist_2\dist_4}{\dist_3+\dist_4}.
\end{equation}
Rearranging gives
\begin{align}\nonumber
\begin{split}
   \interp{\rwlap}(x) &= \frac{1}{A}\left(-\frac{\dist_2 \dist_3}{\dist_1+\dist_2}f_1
   +\dist_3f_2 +\dist_3 f_2
   +\dist_2f_3 +\dist_2 f_3
   -\frac{\dist_2 \dist_3}{\dist_3+\dist_4}f_4\right) \\
   & = \frac{1}{A}\left(-\frac{\dist_2 \dist_3}{\dist_1+\dist_2}f_1
   +\dist_3\frac{\dist_1+\dist_2}{\dist_1+\dist_2}f_2 +\dist_3 f_2
   +\dist_2f_3 +\dist_2 \frac{\dist_3+\dist_4}{\dist_3+\dist_4}f_3
   -\frac{\dist_2 \dist_3}{\dist_3+\dist_4}f_4\right) \\
   & = \frac{1}{A}\left(\frac{\dist_2 \dist_3}{\dist_1+\dist_2}(f_2 - f_1)
   +\left(\frac{\dist_1\dist_3}{\dist_1+\dist_2} +\dist_3\right)f_2
   +\left(\frac{\dist_2\dist_4}{\dist_3+\dist_4} +\dist_2\right)f_3
   +\frac{\dist_2 \dist_3}{\dist_3+\dist_4}(f_3-f_4)\right)
\end{split}
\end{align}
Thus we get
\begin{align}\nonumber
\begin{split}
    \frac{\interp{\rwlap}(x) - f_2}{\dist_2}
    &= \frac{1}{A}\left(\frac{\dist_3}{\dist_1+\dist_2}(f_2 - f_1)
   +\left(\frac{\dist_1\dist_3}{\dist_1+\dist_2} +\dist_3\right)\frac{f_2}{\dist_2}
   +\left(\frac{\dist_4}{\dist_3+\dist_4} +1\right)f_3
   +\frac{\dist_3}{\dist_3+\dist_4}(f_3-f_4)\right) - \frac{1}{A}\frac{Af_2}{\dist_2} \\
&= \frac{1}{A}\left(\frac{\dist_3}{\dist_1+\dist_2}(f_2 - f_1)
   +\left(\frac{\dist_1\dist_3}{\dist_1+\dist_2} +\dist_3 - A\right)\frac{f_2}{\dist_2}
   +\left(\frac{\dist_4}{\dist_3+\dist_4} +1\right)f_3
   +\frac{\dist_3}{\dist_3+\dist_4}(f_3-f_4)\right),
\end{split}
\end{align}
and note that the $f_2$ term simplifies to
\begin{equation}\nonumber
\left(-\frac{\dist_4}{\dist_3+ \dist_4} - 1\right)f_2,
\end{equation}
which results in
\begin{align}
\begin{split}\label{eq:rwproof_est2}
\frac{\interp{\rwlap}(x) - f_2}{\dist_2}
    &= \frac{1}{A}\left(\frac{\dist_3}{\dist_1+\dist_2}(f_2 - f_1)
   +\left(\frac{\dist_4}{\dist_3+\dist_4} +1\right)(f_3 - f_2)
   +\frac{\dist_3}{\dist_3+\dist_4}(f_3-f_4)\right) \\
    &= \frac{1}{A}\left(\frac{\dist_3\dist_1}{\dist_1+\dist_2}\frac{f_2 - f_1}{\dist_1}
   +\left(\frac{\dist_4}{\dist_3+\dist_4} +1\right)(\dist_2+\dist_3)\frac{f_3 - f_2}{\dist_2+\dist_3}
   +\frac{\dist_3\dist_4}{\dist_3+\dist_4}\frac{f_3-f_4}{\dist_4}\right).
\end{split}
\end{align}
This gives the estimate
\begin{align}
\begin{split}\label{eq:rwproof_finalestimate}
\frac{|\interp{\rwlap}(x) - f_2|}{\dist_2}
    &\leq
    \lip{f}
    \frac{1}{A}\left(\frac{\dist_3\dist_1}{\dist_1+\dist_2}
   +\left(\frac{\dist_4}{\dist_3+\dist_4} +1\right)(\dist_2+\dist_3)
   +\frac{\dist_3\dist_4}{\dist_3+\dist_4}\right) \\
   & =
   \lip{f}
   \frac{1}{A}\left(A + \frac{2\dist_3\dist_4}{\dist_3+\dist_4}\right) =   \lip{f}
   \left(1 + 2\frac{\dist_3}{A}\frac{\dist_4}{\dist_3+\dist_4}\right) < 3\lip{f},
\end{split}
\end{align}

where in the last line we used $\dist_3<A$.
\end{proof}

\begin{example}\label{example:rwlip}
To show that \Cref{thm:lipschitz_rwlap_Gx_G} is sharp we must consider a somewhat extreme example. We will use a path graph as before, but this time the points will be highly nonuniform, with $x_1$ and $x_2$ close together, but $x_3$ and $x_4$ far apart. Let $0<h\leq 1$, and consider $x_1=2-h$, $x_2=2$, $x_3=3$, $x_4=3+1/h$, and let $f_1=1-h$, $f_2=1$, $f_3=2$, $f_4=2-1/h$, which always gives $\lip{f}=1$. In fact, in this case the first inequality in \eqref{eq:rwproof_finalestimate} is an equality, giving
\begin{equation}\label{eq:rwlipexampleeqn}
\frac{\interp{\rwlap}(x) - f_2}{\dist_2} = 1 + 2\frac{\dist_3}{A}\frac{\dist_4}{\dist_3+\dist_4},
\end{equation}
where $A$ is from \eqref{eq:Arw}.

The evaluation point $x$ will be chosen carefully to achieve the desired bound. Note that to make \eqref{eq:rwlipexampleeqn} as large as possible, we wish $A$ to be as small as possible. Noting that in this case $\dist_2+\dist_3$ = 1, this means we want the left and right terms of \eqref{eq:Arw} to be small - for the right we simply need $\dist_2$ small, and for the left term we require $\dist_1$ to be much smaller than $\dist_2$. Motivated by this, let $x=x_2+\epsilon$, which gives $\dist_1=h$, $\dist_2 = \epsilon$, $\dist_3=1-\epsilon$, and $\dist_4=1/h$. Note that $\dist_2+\dist_3 = 1$, and let $\epsilon =\sqrt{h}$. Putting these values into \eqref{eq:Arw}, we get $\lim_{h\rightarrow 0}A = 1$ and $\lim_{h\rightarrow 0} \dist_3\dist_4/(\dist_3+\dist_4) = 1$, which from \eqref{eq:rwlipexampleeqn} yields
\begin{equation}\nonumber
\lim_{h\rightarrow 0} \frac{\interp{\rwlap}(x) - f_2}{\dist_2} = 3.
\end{equation}

\noindent See \Cref{subfig:rwlipexample} for plots of some these interpolants over the entire interval $(x_2,x_3)$ for various values of $h$.
\end{example}

\begin{figure}[ht!]
\centering
\begin{subfigure}[b]{0.45\textwidth}
\includegraphics[width=\textwidth]{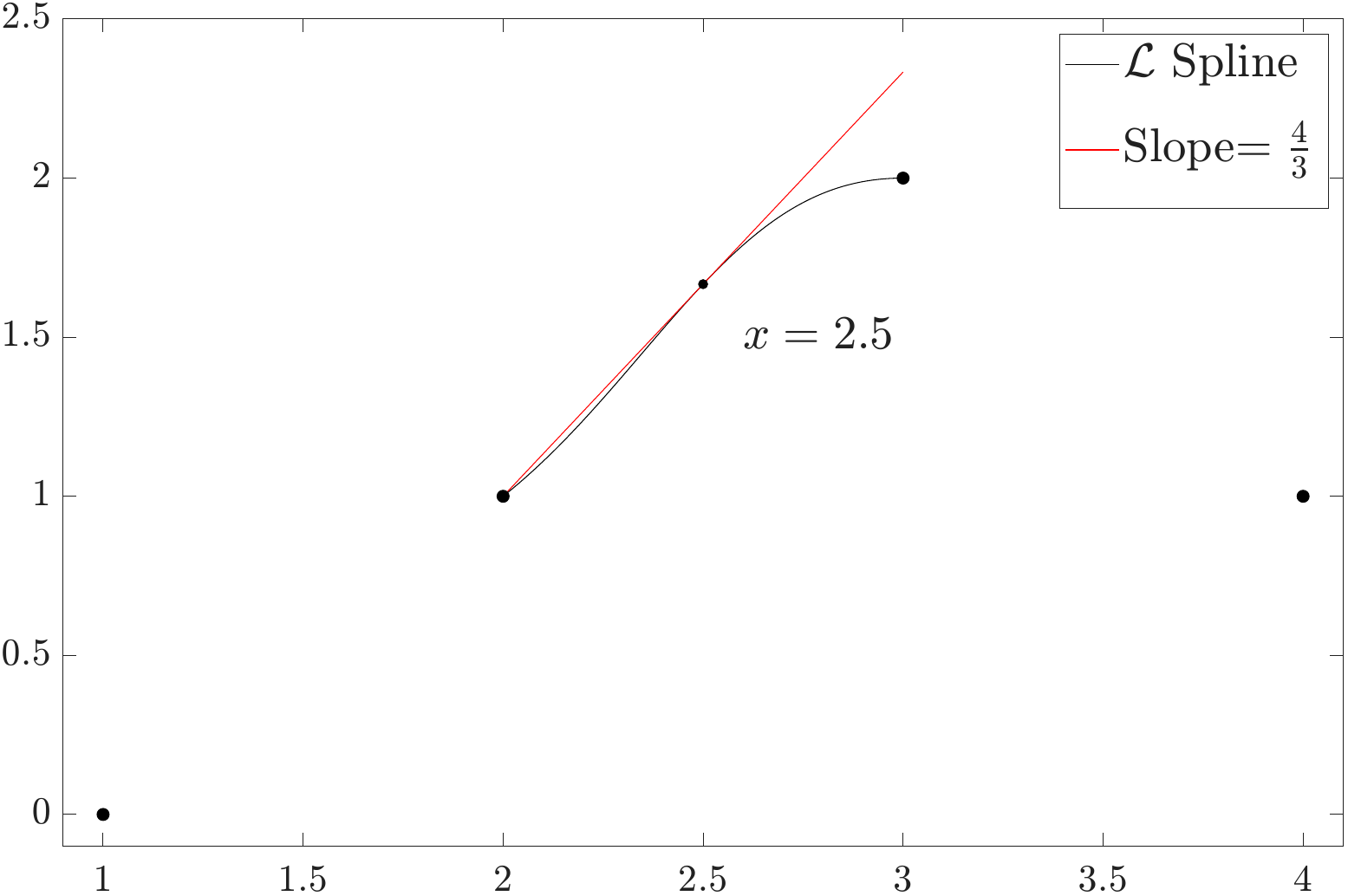}
\caption{The standard Laplacian spline for \Cref{example:stdlip}.}
\label{subfig:stdlipexample}
\end{subfigure}
\hfill
\begin{subfigure}[b]{0.45\textwidth}
\includegraphics[width=\textwidth]{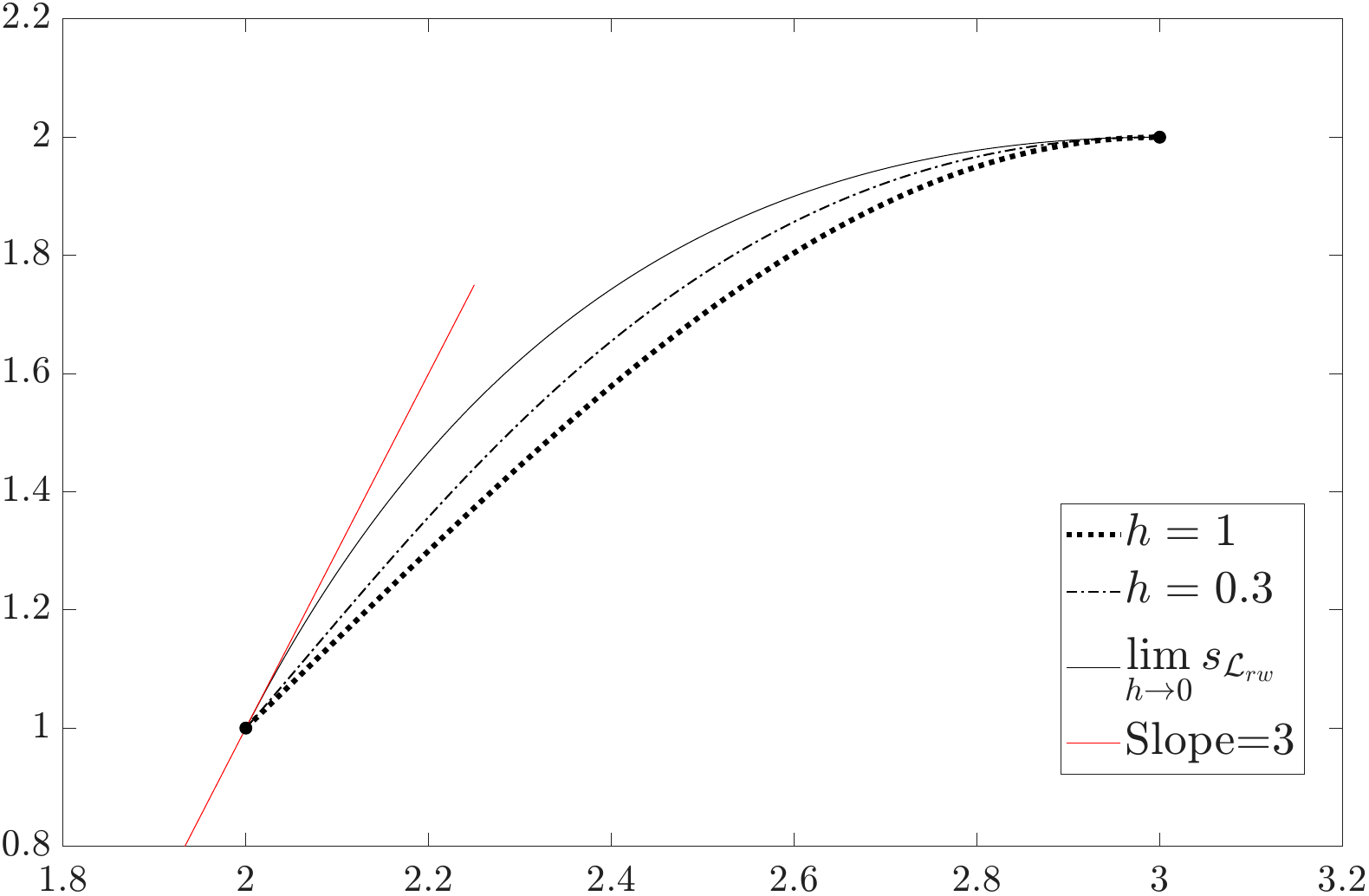}
\caption{The random walk splines for \Cref{example:rwlip}.}
\label{subfig:rwlipexample}
\end{subfigure}
\caption{Examples illustrating some extreme cases for the Lipschitz constants of the spline interpolants on $\graphx$ (when the Lipschitz constant for the graph data is $1$). The standard Laplacian spline $\interp{\lap}$ from \Cref{example:stdlip} is plotted in (a). The results from \Cref{example:rwlip} are in (b), where we see the random walk spline $\interp{\rwlap}$ for various values of $h$. Recall that when $h=1$, the nodes are the same as those in (a), and as $h$ shrinks the left two nodes get closer together while the right two nodes move further apart. }
\label{fig:lipexamples}
\end{figure}

The sharpness construction in \Cref{example:rwlip} is qualitatively different from the corresponding examples for the standard Laplacian and \knn interpolants, relying on a highly nonuniform distribution of data sites. This suggests that the sharp bound for the random walk interpolant appears to be driven by a highly nonuniform sampling geometry. Although the random walk and \knn interpolants have the same worst-case Lipschitz magnification factor of $3$, the random walk bound may be overly pessimistic when the data sites are distributed more uniformly.

\subsection{Infinity norm of one-point interpolation operator}

In this section we briefly investigate stability by discussing the Lebesgue constants for the one point interpolants, and here we deviate from the simple path cases and consider more general graphs. We begin with brief mentions of the $k$NN and standard Laplacian stability estimates, and then discuss the random walk case in more detail.

It is easy to see that the \knn interpolant is very well behaved in this regard, since the interpolation value will always be a convex combination of function values, so we immediately get
\begin{align}
  \abs{\interpkNN(x)}
  \leq
  \|f_X\|_\infty.
\end{align}

The standard Laplacian interpolant is, in general, unbounded in the infinity norm. This can be easily seen in the path case from the last section (see \eqref{eq:stdpathinterp}) by choosing $\rho_1$ and $\rho_4$ small with $\rho_2=\rho_3$ fixed. It is worth noting that this example relies on the data sites being highly nonuniform. We anticipate that if the data site uniformity is controlled, then the Lebesgue constant for the standard Laplacian spline will be bounded. We leave investigating this to future work.

Like the standard Laplacian spline, the random walk spline is defined as the minimizer of an energy functional. One might therefore expect its stability properties to depend on the distribution of the data. However, the next result shows that the random walk spline admits a uniformly bounded Lebesgue constant, independent of the geometry of the data sites.

\begin{theorem}\label{thm:rwthm}
Let $\graph$ be a graph with vertex set $X$, let $x$ be a point in the metric space not from $X$, and let $\graphx$ be a graph with vertex set $\{x\}\cup X$. Then the one-point spline interpolant for the random walk Laplacian satisfies
\begin{equation}\label{eq:rwthm_finalest}
    |\interp{\rwlap}(x)|
    \leq
    \norm{f_X}_\infty + \max_{y\in X} \left|\sum_{v\sim y,v\in X}\frac{\weight{v}{y}}{d_y}(f(y) - f(v))\right|
    \leq
    3\|f_X\|_\infty.
\end{equation}
\end{theorem}

\begin{proof}
From \eqref{eq:rational_function_weighted}, the formula for the interpolant is
\begin{equation}\label{eq:rwthm_interpform}
\interp{\rwlap}(x)
=
\frac{-\lap_x^{\top}D^{-1}\lap_X
f_X}{\lap_x^{\top}D^{-1}\lap_x},
\end{equation}
where the reader will recall that for the one point interpolant $\lap_{\uverts} = \lap_x$ represents the column of the standard graph Laplacian for $\graphx$ associated with $x$, $\lap_{\kverts} = \lap_X$ represents the columns at the data sites, and $D$ is the diagonal degree matrix. Focusing on the denominator, we have
\begin{equation}\nonumber
\parenth{D^{-1}\lap_x}_y
=
\begin{cases}
1
& y=x\\
-\frac{\weight{x}{y}}{d_y}
& y\in X\\
\end{cases},
\end{equation}
where $d_y$ denotes vertex degree, \eqref{eq:vertex_degree_definition}.
This gives
\begin{equation}\label{eq:rwthm_denomest}
\lap_x^\top \parenth{D^{-1}\lap_x}
=
d_x
+
\sum_{y\sim x}
\frac{\weight{x}{y}^2}{d_y}
\geq
d_x
\end{equation}

Regarding the numerator, note that the entries of $D^{-1}\lap_X f_X$ are given by
\begin{equation}\nonumber
\parenth{D^{-1}\lap_X f_X}_y
=
\begin{cases}
f(y)
-
\displaystyle\sum_{v\sim y, v\in X}\frac{\weight{v}{y}}{d_y} f(v)
& y\in X\\
-\displaystyle \sum_{v\sim y, v\in X}\frac{\weight{v}{y}}{d_y} f(v)
& y= x.
\end{cases}
\end{equation}
This leads to
\begin{align}
\lap_x^\top \parenth{D^{-1}\lap_X f_X}
&=
d_x
\parenth{-\sum_{v\sim x, v\in X}\frac{\weight{v}{x}}{d_x} f(v)}
+
\sum_{y\in X}
-\weight{x}{y}
\parenth{
f(y) -\sum_{v\sim y, v\in X}\frac{\weight{v}{y}}{d_y} f(v) }\nonumber\\
&=
-\sum_{v\sim x, v\in X} \weight{v}{x} f(v)
+
\sum_{y\in X}
-\weight{x}{y}
\parenth{ f(y) -
\sum_{v\sim y, v\in X}\frac{\weight{v}{y}}{d_y} f(v) }\nonumber\\
&=
-\sum_{y\sim x, y\in X} \weight{y}{x} f(y)
+
\sum_{y\in X}
-\weight{x}{y}
\parenth{ f(y) -
\sum_{v\sim y, v\in X}\frac{\weight{v}{y}}{d_y} f(v) }\label{eq:rwthm_numeq1},
\end{align}
where in the last line we simply relabeled indices. Note that part of the rightmost term in \eqref{eq:rwthm_numeq1} can be rewritten as
\begin{align}\nonumber
\begin{split}
f(y) -\sum_{v\sim y, v\in X}\frac{\weight{v}{y}}{d_y} f(v)
&=
\frac{d_yf(y)}{d_y} -\sum_{v\sim y, v\in X}\frac{\weight{v}{y}}{d_y}f(v)\\
&=
\frac{1}{d_y}\parenth{w(x,y)+\sum_{v\sim y,v\in X}\weight{v}{y}}f(y) -\sum_{v\sim y, v\in X}\frac{\weight{v}{y}}{d_y}f(v)\\
&=
\frac{\weight{x}{y}}{d_y}f(y) + \sum_{v\sim y, v\in X}\frac{\weight{v}{y}}{d_y}(f(y) - f(v)).
\end{split}
\end{align}
Inserting this into \eqref{eq:rwthm_numeq1} gives
\begin{equation}\nonumber
\lap_x^\top \parenth{D^{-1}\lap_X f_X}
=
-\sum_{y\sim x, y\in X} \parenth{\weight{y}{x} + \frac{\weight{x}{y}^2}{d_y}}f(y)
-
\sum_{y\in X}\weight{x}{y}\sum_{v\sim y, v\in X}\frac{\weight{v}{y}}{d_y}(f(y) - f(v)).
\end{equation}
Bounding this in absolute value gives
\begin{equation}\nonumber
|\lap_x^\top \parenth{D^{-1}\lap_X f_X}|
\leq
\parenth{d_x + \sum_{y\sim x, y\in X} \frac{\weight{x}{y}^2}{d_y}}\norm{f_X}_\infty
+
d_x  \max_{y\in X} \left|\sum_{v\sim y,v\in X}\frac{\weight{v}{y}}{d_y}(f(y) - f(v))\right|,
\end{equation}
and combining this with \eqref{eq:rwthm_interpform} and $\eqref{eq:rwthm_denomest}$ gives
\begin{equation}\nonumber
    |\interp{\rwlap}(x)|
    \leq
    \norm{f_X}_\infty + \max_{y\in X} \left|\sum_{v\sim y,v\in X}\frac{\weight{v}{y}}{d_y}(f(y) - f(v))\right|  \leq 3\norm{f_X}_\infty.
\end{equation}
\end{proof}
We note that \eqref{eq:rwthm_finalest} includes the estimate
\begin{equation}\nonumber
    |\interp{\rwlap}(x)|
    \leq
    \norm{f_X}_\infty + \max_{y\in X} \left|\sum_{v\sim y,v\in X}\frac{\weight{v}{y}}{d_y}(f(y) - f(v))\right|.
\end{equation}
Consequently, when the sampled function is smooth the interpolation value should be much smaller than the worst-case bound $3\norm{f_X}_\infty$.

\section{Experiments on Interpolation}
\label{sec:experiments}

In our experiments\footnote{Our experiments were conducted in Python, and our code will be available on ArXiv.
},
we consider the manifold of positive definite matrices whose geometry
is based on the Cholesky decomposition, \cite{lin2019riemannian}.  These experiments assume that we have a finite collection of positive definite matrices, and we want to estimate some function in a neighborhood of those matrices.

We recall some notation and definitions from \cite{lin2019riemannian}, and we refer the reader there for additional information on the complete manifold structure. The strictly lower triangular part of a matrix $A$ is denoted
$\floor{A}$. The diagonal matrix extracted from $A$ is
$\mathbb{D}(A)$. The Frobenius norm of $A$ is denoted
$\norm{A}_F$.
For two positive definite, Cholesky-factored matrices $A=RR^\top$ and $B=SS^\top$, the geodesic distance between them is
\begin{equation}\nonumber
\dist_\metspace(A,B)
=
\set{
\norm{\floor{R}-\floor{S} }_F^2
+
\norm{ \log \mathbb{D}(R) -\log \mathbb{D}(S) }_F^2
}^{1/2}.
\end{equation}
Note that geodesics are computed in terms of Cholesky factors.
A geodesic starting at $A$, in the direction $\Gamma$ (a lower triangular matrix), is derived from the parametrized Cholesky factor
\begin{equation}\nonumber
\floor{R} +t \floor{\Gamma} + \mathbb{D}(R)
\exp \set{ t \mathbb{D}(\Gamma) \mathbb{D}(R)^{-1}}.
\end{equation}

\subsection{Path Interpolation}
\label{subsec:experiment_path}

Here, we consider interpolation on the path graph (\Cref{fig:graph_example_path})
embedded on a geodesic.
The underlying metric space is the manifold of symmetric positive definite matrices of size
$200\times 200$, and we interpolate the radial function
\begin{equation}\label{eq:target}
f_B(A) = \cos(\dist_\metspace(A,B)),
\end{equation}
with \emph{center} $B$, after adding noise to the data. We compare the spline and
$k$NN interpolants.

For a single experiment, we start by randomly generating the function center $B$, an initial vertex/point $x_1$, and a direction $\Gamma$ defining the geodesic. The remaining vertices are equally spaced along the geodesic with spacing 0.25. A fine grid between the central vertices $x_2,x_3$ is defined for evaluation points.
Before interpolation, we add random noise to the function data at the vertices $X$ of the orignal graph, from a uniform distribution on
$[-0.5\gamma,0.5\gamma]$,
where $\gamma$ denotes a chosen noise level. After interpolating the values on the fine grid, we compute the MSE between the true function values and the interpolated values.

For each noise level $\gamma$, we run 50 iterations with random points and random noise on the function values. The average error
for each noise level is recorded and plotted as a function of
$\gamma$ in in \Cref{fig:error_path}. Note that in this experiment,
the standard and random walk splines gave similar results, so of the two, only
the random walk spline is included.

\subsection{Planar Interpolation}
\label{subsec:experiment_triangle}

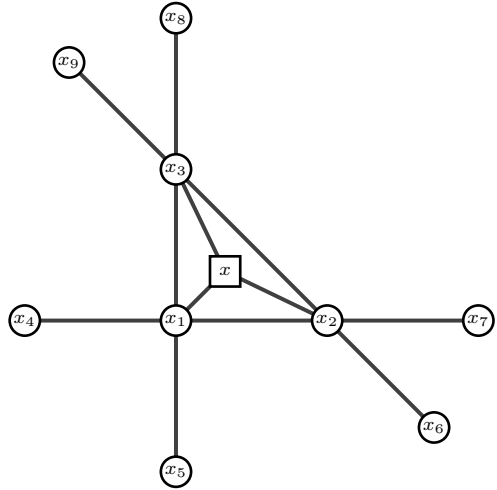
\begin{figure}[ht!]
\centering
\input{figures/planar1}
\caption{\textbf{Planar graph}. Circles represent vertices where function values are known. The rectangle represents the vertex where we want to interpolate the function value.
}
\label{fig:graph_example_triangulation}
\end{figure}

This experiment is similar to the path case, except we are now on  a two-dimensional submanifold and working on a planar graph (\Cref{fig:graph_example_triangulation}). To construct the graph, we randomly generate the vertex $x_1$ and a random direction $\Gamma_1$. From this, we define $x_4,x_2,x_7$, where the four points are equally spaced on a geodesic, with spacing 0.25. We repeat this process for a second randomly generated direction from $x_1$ to $x_3$
to obtain $x_5$ and $x_8$. Finally, we compute the direction from $x_2$ to $x_3$ to define a third geodesic and fill in the remaining graph vertices. Distances between adjacent vertices
$\dist_\metspace(x_i,x_j)$ are all 0.25, except
$\dist_\metspace(x_2,x_3)$. The evaluation points live on a fine grid in the triangular interior to $x_1,x_2,x_3$. Noise is added to a randomly placed function \eqref{eq:target}
and interpolation errors are averaged as in the path case. Results are displayed in \Cref{fig:error_triangle}.

\begin{figure}
     \centering
     \begin{subfigure}[b]{0.48\textwidth}
         \centering
         \includesvg[width=\textwidth]{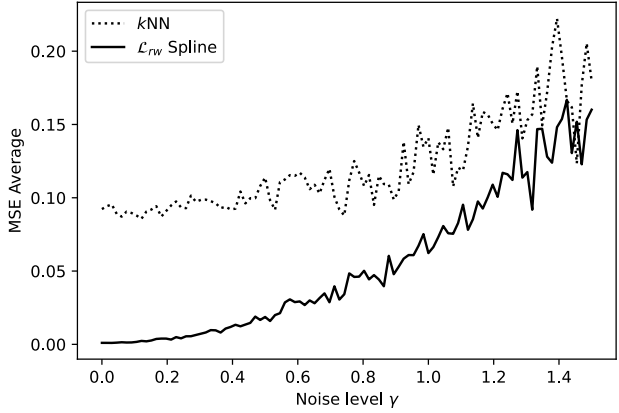}
         \caption{Path graph, see  \Cref{fig:graph_example_path}.}
         \label{fig:error_path}
     \end{subfigure}
     \hfill
     \begin{subfigure}[b]{0.48\textwidth}
         \centering
         \includesvg[width=\textwidth]{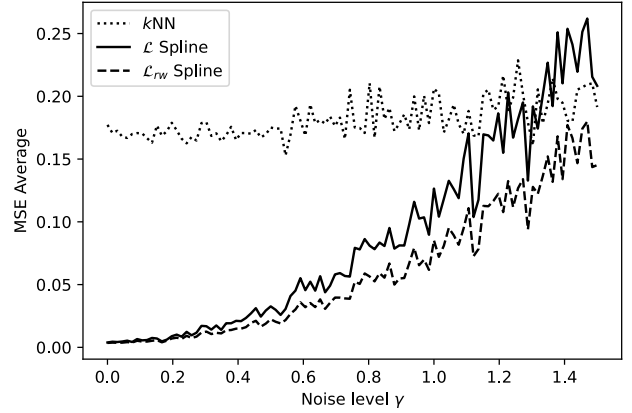}
         \caption{Planar graph, see \Cref{fig:graph_example_triangulation}.}
         \label{fig:error_triangle}
     \end{subfigure}
        \caption{\textbf{Interpolation error} for smooth data with added noise on graphs in the manifold of positive semi-definite matrices, where $\gamma$ indicates the level of the noise.}
        \label{fig:experiment_error}
\end{figure}

\subsection{Discussion}\label{subsec:discussion}

 At low noise levels, the graph-based interpolants in both examples achieve lower errors than $k$NN interpolation. As the noise level increases, the averaging effect of $k$NN becomes advantageous, and eventually the error for the spline interpolants approaches or exceeds that of $k$NN. These behaviors are consistent with the spline methods incorporating local smoothness through energy minimization, while $k$NN relies primarily on local averaging. The random walk spline, however, stands out in a way that is not as readily predicted from this distinction. In the planar experiment, for example, at approximately $\gamma =1.3$, the standard Laplacian spline error starts to grow beyond that of the $k$NN interpolant. However, the random walk spline performs comparably to or better than $k$NN in our experiments for all noise levels considered. These examples, along with the stability result in \Cref{thm:rwthm}, suggest that the random walk spline may offer a useful compromise between capturing trends in the data and the noise-averaging behavior of $k$NN interpolation, although further experiments would be needed to assess this behavior more broadly.

\bibliographystyle{plain}
\bibliography{refs}

\appendix

\section{Formulas for Graph Interpolants}
\label{sec:formulas}

Using the normal equations \eqref{eq:rational_function}, we find explicit formulas for the graph interpolants in terms of the distances between the vertices and function values. These are primarily used for our theoretical analysis.

\subsection{First-order Derivatives -- Short Path Subgraph}

For first-order derivatives, we are mainly concerned with interpolation along an edge in $\graph$, so we consider $\subgraphx$ as in \Cref{fig:graph_example_shortpath}.  The distances are denoted $\dist_1=\dist_{\subgraphx}(x,x_1)$ and $\dist_2=\dist_{\subgraphx}(x,x_2)$. we assume the matrix $M$ has the differences aligned in the same direction for the two edges. The formulas are

\begin{equation}\nonumber
\interp{\fodlapfactor}(x)
=
\frac{f_1 \dist_2 +f_2\dist_1}{\dist_1 +\dist_2},
\end{equation}
\begin{equation}\nonumber
\interp{\fodgrad}(x)
=
\frac{f_1 \dist_2^2 +f_2\dist_1^2}{\dist_1^2 +\dist_2^2}.
\end{equation}

\subsection{Laplacian -- Path Subgraph}
\label{subsec:formula_path_subgraph_laplacian}

For the Laplacian interpolant $\interp{\lap}$, the relevant subgraph
$\subgraphx$ contains neighbors of neighbors.  If we interpolate along an edge in a path in $\graph$, the subgraph $\subgraphx$ is \Cref{fig:graph_example_path}.
The distances are denoted
$\dist_1=\dist_{\subgraphx}(x_1,x_2)$,
$\dist_2=\dist_{\subgraphx}(x_2,x)$,
$\dist_3=\dist_{\subgraphx}(x_3,x)$,
$\dist_4=\dist_{\subgraphx}(x_3,x_4)$.
In this case, if we index the rows and columns as $x_1,x_2,x,x_3,x_4$, the standard Laplacian becomes
\begin{equation}\label{eq:stdlap_path}
\lap
=
\begin{bmatrix}
\frac{1}{\dist_1} & -\frac{1}{\dist_1} & 0 & 0 & 0\\
-\frac{1}{\dist_1} & \frac{\dist_1+\dist_2}{\dist_1\dist_2} & -\frac{1}{\dist_2} &0 & 0\\
0 & -\frac{1}{\dist_2} &
\frac{\dist_2+\dist_3}{\dist_2\dist_3}& -\frac{1}{\dist_3} & 0\\
0 & 0 & -\frac{1}{\dist_3} &
\frac{\dist_3+\dist_4}{\dist_3\dist_4}& -\frac{1}{\dist_4}\\
0 & 0 & 0 & -\frac{1}{\dist_4} & \frac{1}{\dist_4}
\end{bmatrix}
\end{equation}
Note that the zeros in $\lap_{\uverts}$ (the column associated with $x$) mean that we need only concern ourselves with rows $x_2,x,x_3$ of the Laplacian when considering the normal equation \eqref{eq:rational_function}. We multiply by the common factor $\dist_2\dist_3$ and restrict to these rows to get
\begin{equation}\nonumber
(\dist_2\dist_3)\lap_{\uverts}
=
\begin{bmatrix}
 - \dist_3  \\
 \dist_2+\dist_3\\
 - \dist_2 \\
\end{bmatrix}
\quad\mbox{and}\quad
(\dist_2\dist_3)
\lap_{\kverts}
=
\begin{bmatrix}
-\frac{\dist_2\dist_3}{\dist_1} & \frac{(\dist_1 +\dist_2)\dist_3}{\dist_1}  &0 & 0\\
0 & -\dist_3  & - \dist_2 & 0\\
0 & 0 &  \frac{(\dist_3+\dist_4)\dist_2}{\dist_4} & -\frac{\dist_2\dist_3}{\dist_4}\\
\end{bmatrix}
\end{equation}
From this, we can use the normal equations \eqref{eq:rational_function} to deduce the formula
\begin{align}\begin{split}\label{eq:pathlapinterp_appendix}
\interp{\lap}(x)
&=
\frac{1}{\dist_2^2+(\dist_2+\dist_3)^2+\dist_3^2}
\begin{bmatrix}
f_1 & f_2 & f_3 & f_4
\end{bmatrix}
\begin{bmatrix}
-\frac{\dist_2\dist_3^2}{\dist_1}\\
\frac{\dist_2\dist_3^2}{\dist_1}
+  (\dist_2+2\dist_3)\dist_3\\
\frac{\dist_3\dist_2^2}{\dist_4}
+ (2\dist_2+\dist_3)\dist_2\\
- \frac{\dist_2^2\dist_3}{\dist_4}
\end{bmatrix},
\end{split}\end{align}
where $f_i = f(x_i)$.
\subsection{Random Walk -- Path Subgraph}
\label{subsec:formula_path_subgraph_rwlap}

Using $\rwlap = D^{-1}\lap$, the normal equations
can be recast as \eqref{eq:rational_function_weighted}.
As before, we may restrict to the $x_2,x,x_3$ rows of the standard Laplacian as in \eqref{eq:stdlap_path}, and note that the relevant portion of the inverse of the degree matrix is
\begin{equation}\nonumber
D^{-1} =
\begin{bmatrix}
0 & \frac{\dist_1\dist_2}{\dist_1+\dist_2} & 0 &0 & 0\\
0 & 0 &
\frac{\dist_2\dist_3}{\dist_2+\dist_3}& 0 & 0\\
0 & 0 & 0 &
\frac{\dist_3\dist_4}{\dist_3+\dist_4}& 0\\
\end{bmatrix}
\end{equation}
Using this with \eqref{eq:stdlap_path}, get the following identities
\begin{equation}\nonumber
\lap_U
=
\frac{1}{\dist_2\dist_3}\begin{bmatrix}
 - \dist_3  \\
 \dist_2+\dist_3\\
 - \dist_2
\end{bmatrix},
\quad
D^{-1}\lap_U=
\begin{bmatrix}
-\frac{\dist_1}{\dist_1+\dist_2} \\
1\\
 - \frac{\dist_4}{\dist_3+\dist_4}
\end{bmatrix},\quad
\lap_{\kverts}
=\frac{1}{\dist_2\dist_3}
\begin{bmatrix}
-\frac{\dist_2\dist_3}{\dist_1} & \frac{(\dist_1 +\dist_2)\dist_3}{\dist_1}  &0 & 0\\
0 & -\dist_3  & - \dist_2 & 0\\
0 & 0 &  \frac{(\dist_3+\dist_4)\dist_2}{\dist_4} & -\frac{\dist_2\dist_3}{\dist_4}
\end{bmatrix}
\end{equation}
Now, we compute the components of the normal equations, and we get
\begin{align*}
(\dist_2\dist_3)
\lap_{\uverts}^\top
D^{-1}
\lap_{\uverts}
&=
\frac{\dist_1\dist_3}{\dist_1+\dist_2}
+
(\dist_2+\dist_3)
+
\frac{\dist_2\dist_4}{\dist_3+\dist_4},
\end{align*}
\begin{align*}
(\dist_2\dist_3)
\lap_{\uverts}^\top
D^{-1}
\lap_{\kverts}
&=
\begin{bmatrix}
\frac{\dist_2\dist_3}{\dist_1+\dist_2} & -2\dist_3 & -2\dist_2 & \frac{\dist_2\dist_3}{\dist_3+\dist_4}
\end{bmatrix}^\top.
\end{align*}
Therefore the interpolant is the rational function
\begin{align}\begin{split}\label{eq:pathrwlapinterp_appendix}
\interp{\rwlap}(x)
&=
-
\frac{1}{\frac{\dist_1\dist_3}{\dist_1+\dist_2}
+
(\dist_2+\dist_3)
+
\frac{\dist_2\dist_4}{\dist_3+\dist_4}}
\begin{bmatrix}
f_1 & f_2 & f_3 & f_4
\end{bmatrix}
\begin{bmatrix}
\frac{\dist_2\dist_3}{\dist_1+\dist_2}\\
-2\dist_3 \\
-2\dist_2\\
\frac{\dist_2\dist_3}{\dist_3+\dist_4}
\end{bmatrix}.
\end{split}\end{align}

\subsection{$k$NN, Path Subgraph}\label{sec:kNN}

The $k$NN interpolant for the path subgraph in \Cref{fig:graph_example_path} is simply a weighted average with weights determined by the distance of the nodes from $x$. This yields the formula below.
\begin{align}\label{eq:kNNpathformula}
\interpkNN(x)
&=
\frac{1}{(\dist_1+\dist_2)^{-1}+ \dist_2^{-1}+\dist_3^{-1} +(\dist_3+\dist_4)^{-1}}
\parenth{
(\dist_1+\dist_2)^{-1} f_1
+
\dist_2^{-1} f_2
+
\dist_3^{-1}f_3
+
(\dist_3+\dist_4)^{-1}f_4
}\nonumber\\
&=
\frac{1}{
\frac{\dist_2\dist_3}{\dist_1+\dist_2}
+ \dist_3+\dist_2
+\frac{\dist_2\dist_3}{\dist_3+\dist_4}}
\parenth{
\frac{\dist_2\dist_3}{\dist_1+\dist_2} f_1
+
\dist_3 f_2
+
\dist_2 f_3
+
\frac{\dist_2\dist_3}{\dist_3+\dist_4} f_4
}
\end{align}
Note the similarities and subtle differences to the formula for the random walk spline interpolant in \eqref{eq:pathrwlapinterp_appendix}.

\end{document}

%% file: macros.tex
\DeclareMathOperator*{\argmin}{arg\,min}
\providecommand{\abs}[1]{\left\lvert#1\right\rvert}
\providecommand{\norm}[1]{\left\rVert#1\right\rVert}
\providecommand{\set}[1]{ \left\{ #1  \right\}  }

\providecommand{\parenth}[1]{\left( #1 \right) }

\providecommand{\floor}[1]{\left\lfloor #1 \right\rfloor }

\providecommand{\dist}{\rho}

\providecommand{\vertices}{\mathcal{V}}
\providecommand{\uverts}{U}
\providecommand{\kverts}{K}

\providecommand{\lap}{\mathcal{L}}
\providecommand{\lap}{\bar{\mathcal{L}}} 
\providecommand{\nlap}{\lap_{n}}

\providecommand{\rwlap}{\lap_{rw}}

\providecommand{\wspd}{\ell_2(\vertices,D)}

\providecommand{\interp}[1]{s_{#1}} 
\providecommand{\interplinear}{s_{lin}} 
\providecommand{\interpkNN}{s_{\mbox{$k$NN}}} 
\providecommand{\knn}{$k$NN } 
\providecommand{\lip}[1]{C\left(#1\right)} 

\newcommand{\graph}{\mathcal{G}} 
\newcommand{\graphx}{\graph_x} 
\newcommand{\subgraphx}{\mathcal{H}_x} 
\newcommand{\verticesx}{\vertices_x} 
\newcommand{\edges}{E} 

\newcommand{\metspace}{{\mathcal{X}}} 
\newcommand{\localnodes}{X} 
\newcommand{\nodes}{X} 
\newcommand{\derivative}{\Psi} 

\newcommand{\fodgrad}{\nabla} 
\newcommand{\fodlapfactor}{D_r} 
\newcommand{\weight}[2]{w(#1,#2)}

%% file: figures/graph.tikzstyles

\tikzstyle{circle}=[fill=white, draw=black, shape=circle, scale=.7]
\tikzstyle{shadedcircle}=[fill={rgb,255: red,170; green,170; blue,170}, draw=black, shape=circle, scale=.8]
\tikzstyle{blueretangle}=[fill=blue, draw=black, shape=rectangle]
\tikzstyle{rectangle}=[fill=white, draw=black, shape=rectangle]

\tikzstyle{dashes}=[-, fill=white, draw=black, dashed]
\tikzstyle{bold}=[-, fill=white, line width=.6mm]

%% file: figures/planar1.tex
\tikzmath{
   \gap=2;
   \slope=-1;
   \xc = 0.65;
   \yc = 0.65;
   \xx = -\gap/sqrt(2);
   \yy = \gap+sqrt(2);
   \xxx = 2+sqrt(2);
   \yyy = -sqrt(2);
}
\begin{tikzpicture}
\Vertex[x=\xc,y=\yc,label=$x$,shape=rectangle,opacity=0,size=\vertexsize]{x}  

\Vertex[x=0,y=0,label=$x_1$,opacity=0,size=\vertexsize]{x1}
\Vertex[x=-\gap,y=0,
label=$x_4$,opacity=0,size=\vertexsize]{x4}
\Vertex[x=\gap,y=0,
label=$x_2$,opacity=0,size=\vertexsize]{x2}
\Vertex[x=2*\gap,y=0,
label=$x_7$,opacity=0,size=\vertexsize]{x7}

\Vertex[x=0,y=-\gap,
label=$x_5$,opacity=0,size=\vertexsize]{x5}
\Vertex[x=0,y=\gap,
label=$x_3$,opacity=0,size=\vertexsize]{x3}
\Vertex[x=0,y=2*\gap,
label=$x_8$,opacity=0,size=\vertexsize]{x8}

\Vertex[x=\xx,y=\yy,
label=$x_9$,opacity=0,size=\vertexsize]{x9}
\Vertex[x=\xxx,y=\yyy,
label=$x_6$,opacity=0,size=\vertexsize]{x6}



%

%

\Edge[](x)(x1)
\Edge[](x)(x2)
\Edge[](x)(x3)

\Edge[](x1)(x2)
\Edge[](x1)(x3)
\Edge[](x2)(x3)

\Edge[](x1)(x4)
\Edge[](x1)(x5)

\Edge[](x2)(x6)
\Edge[](x2)(x7)

\Edge[](x3)(x8)
\Edge[](x3)(x9)

\end{tikzpicture}